\documentclass{amsart}

\PassOptionsToPackage{ngerman,main=english,shorthands=off}{babel}
\usepackage[dvipsnames]{xcolor}
\usepackage{amsmath, amsthm, amsfonts, mathtools, tikz-cd, amssymb}
\usepackage{enumitem, comment}

\usepackage{tikz}
\usepackage{quiver}
\usetikzlibrary{positioning}
\usetikzlibrary{decorations.markings}
\tikzset{
    double line with arrow/.style args={#1,#2}{
        decorate,decoration={markings,%
            mark=at position 0 with {\coordinate (ta-base-1) at (0,1pt);
                \coordinate (ta-base-2) at (0,-1pt);},
            mark=at position 1 with {\draw[#1] (ta-base-1) -- (0,1pt);
                \draw[#2] (ta-base-2) -- (0,-1pt);
            }
        }
    }
}
\tikzset{Equal/.style={-,double line with arrow={-,-}}}

\usepackage{babel}
\usepackage[utf8]{inputenc}
\usepackage[T1]{fontenc}
\usepackage{lmodern}
\usepackage[babel,
    protrusion=true,
    expansion=true,
    kerning=true,
    tracking=true,
]{microtype}
\SetTracking{encoding={*}, shape=sc}{20} 
\newcommand*{\mathup}[1]{\mathrm{#1}}

\AddToHook{cmd/tableofcontents/before}{%
    \microtypesetup{protrusion=false}
}
\AddToHook{cmd/tableofcontents/after}{%
    \microtypesetup{protrusion=true}
}

\usepackage{needspace}
\newcommand*{\AvoidPageBreak}{\Needspace*{0pt}}

\usepackage[shortcuts]{extdash}

\usepackage[all]{nowidow}
\usepackage{csquotes}
\usepackage{fnpct}
\mathtoolsset{mathic=true}

\usepackage[
    pagebackref,
    urlcolor=black,
    linkcolor=blue,
    citecolor=Green,
    allbordercolors=blue,
    citebordercolor=Green
]{hyperref}
\hypersetup{hypertexnames=true}

\usepackage[
    alphabetic,
    backrefs,
    msc-links
]{amsrefs}
\usepackage[thmtools-compat]{keytheorems}

\newkeytheoremstyle{linkedrestate}{%
    headformat={%
        \protect\IfRestatingTF{%
            \protect\hypertarget{orig-#1-#2}{}%
            \protect\hyperlink{restated-#1-#2}{\NAME\thmnumber{ #2}}%
            \NOTE%
        }{%
            \protect\hypertarget{restated-#1-#2}{}%
            \NAME\thmnumber{ #2}%
            \NOTE%
        }%
    },
}

\usepackage{zref-clever}
\zcsetup{
    cap,
    nameinlink,
    abbrev=true,
    countertype={
        defn=definition,
        thm=theorem,
        thmalpha=theorem,
        prop=proposition,
        lem=lemma,
        cor=corollary,
        coralpha=corollary,
        rem=remark
    }
}
\zcRefTypeSetup{condenumi}{
    Name-sg = Condition ,
    name-sg = condition ,
    Name-pl = Conditions ,
    name-pl = conditions ,
}

\newcommand{\cref}[2][]{\zcref[#1]{#2}}

\newcommand{\Cref}[2][]{\zcref[S,#1]{#2}}

\theoremstyle{definition}
\newtheorem{defn}{Definition}[section]
\newtheorem{example}[defn]{Example}

\newtheorem*{question*}{Question}
\newtheorem*{conjecture*}{Conjecture}
\newtheorem{rem}[defn]{Remark}

\newtheorem{construction}[defn]{Construction}
\theoremstyle{plain}
\theoremstyle{linkedrestate}
\newtheorem{thm}[defn]{Theorem}
\newtheorem{lem}[defn]{Lemma}
\newtheorem{prop}[defn]{Proposition}
\newtheorem{cor}[defn]{Corollary}

\newtheorem{thmalpha}{Theorem}

\newtheorem{coralpha}[thmalpha]{Corollary}

\numberwithin{equation}{section}

\newlist{condenum}{enumerate}{1}
\setlist[condenum]{label=\upshape(\roman*)}
\setlist[enumerate]{label=\upshape(\arabic*)}

\newlist{refenumImpl}{enumerate}{2}
\zcsetup{
    counterresetby = {
        refenumImplii  = refenumImpli ,
        refenumImpliii = refenumImplii ,
        refenumImpliv  = refenumImpliii ,
    }
}
\newcommand{\RefEnumSetCounter}[1]{%
    \setlist[refenumImpl,1]{label=\upshape(\arabic*),ref=\csname the#1\endcsname \,\upshape(\arabic*)}%
    \setlist[refenumImpl,2]{label=\upshape(\alph*), ref=\csname the#1\endcsname \,\upshape(\arabic{refenumImpli}-\alph*)}%
}
\newcommand{\RefEnumPatchCounterType}[1]{%
    \expanded{%
        \noexpand\zcsetup{%
            countertype = {
                refenumImpli   = #1 ,
                refenumImplii  = #1 ,
                refenumImpliii = #1 ,
                refenumImpliv  = #1 ,
            }}%
    }%
}

\ExplSyntaxOn
\NewDocumentCommand \IfRefenumCounter { m m m }{
    \str_case:enTF {#1} {
        {refenumImpli} {}
        {refenumImplii} {}
        {refenumImpliii} {}
        {refenumImpliv} {}
    }{#2}{#3}
}
\ExplSyntaxOff

\makeatletter
\newenvironment{refenum}[1][]{%
    \IfRefenumCounter{\@currentcounter}{%
    }{%
        \edef\refenum@parentcounter{\@currentcounter}%
        \edef\refenum@parenttype{\zref@getcurrent{zc@type}}%
        \RefEnumSetCounter{\refenum@parentcounter}%
        \RefEnumPatchCounterType{\refenum@parenttype}%
    }
    \begin{refenumImpl}[#1]%
}{%
    \end{refenumImpl}%
}
\makeatother

\ExplSyntaxOn
\cs_new_protected:Nn \gaussler_extract_subref:n {
    \__gaussler_extract_last:n #1
}
\cs_new:Nn \__gaussler_extract_last:n {
    \__gaussler_extract_last:w #1 () \q_stop
}
\cs_generate_variant:Nn \gaussler_extract_subref:n {e}
\cs_new:Npn \__gaussler_extract_last:w #1 (#2) #3 \q_stop {
    \tl_if_empty:nTF {#2} {\textbf{??}} {(#2)}
}
\NewDocumentCommand{\localref}{m}{
    \hyperref[#1]{ \textup { \gaussler_extract_subref:e { \getrefnumber{#1} } } }
}
\ExplSyntaxOff

\newcommand{\calG}{\mathcal{G}}

\newcommand{\defq}{\coloneqq}

\newcommand{\onto}{\ensuremath\twoheadrightarrow}

\newcommand{\bs}{\backslash}

\newcommand{\colim}{\operatorname*{colim}}

\newcommand{\id}{\operatorname{id}}

\NewDocumentCommand{\FP}{E{_}{{}}}{%
    \mathup{FP}\IfNoValueTF{#1}{}{_{\mkern-2mu#1}}%
}
\NewDocumentCommand{\F}{E{_}{{}}}{%
    \mathup{F}\IfNoValueTF{#1}{}{_{\mkern-2mu#1}}%
}

\makeatletter
\ExplSyntaxOn
\cs_new:Npn \start_is_cs:nTF #1 {
    \tl_if_empty:nTF { #1 }
    { \use_ii:nn }
    {
        \exp_args:Ne \token_if_cs:NTF { \tl_head:n { #1 } }
    }
}
\NewDocumentCommand{\FirstTokenCommandTF}{ m m m }{
    \start_is_cs:nTF { #1 } { #2 } { #3 }
}
\ExplSyntaxOff

\newbox\usefulbox
\newcommand{\getslant}[1]{\strip@pt\fontdimen1 #1}
\newcommand{\fillKern}[1]{%
    \FirstTokenCommandTF{#1}{%
        \mkern-0.5mu
    }{%
        \setbox\usefulbox=\hbox{$\m@th\scriptstyle #1$}%
        \dimen@ \getslant\the\scriptfont\symletters \ht\usefulbox%
        \kern-\dimen@%
    }%
    #1%
}
\makeatother
\newcommand{\verythinspace}{%
    \ifmmode%
        \mskip0.5\thinmuskip%
    \else%
        \ifhmode%
            \kern0.08334em%
        \fi%
    \fi%
}
\newcommand{\delimiterSpace}{%
    \mathchoice{\verythinspace}{}{}{}%
}

\usepackage{mleftright}
\makeatletter
\newcommand*{\norm}[1]{%
    \if@display%
        \mathord{\@norm{\displaystyle}{\delimiterSpace #1 \delimiterSpace}}%
    \else%
        \mathord{\@norm{}{\delimiterSpace #1 \delimiterSpace}}%
    \fi%
}
\newcommand*{\@norm}[2]{%
    \@autodelimiterpair{#1}{#2}{\lVert}{\rVert}%
}
\newcommand*{\abs}[1]{%
    \mathinner{\@abs{}{#1}}%
}
\newcommand*{\@abs}[2]{%
    \@autodelimiterpair{#1}{#2}{\lvert}{\rvert}%
}
\newcommand*{\@autodelimiterpair}[4]{%
    \sbox0{$#1\vcenter{}$}%
    \dimen0=\ht0 %
    \sbox0{$#1#2$}%
    \dimen2=\dimexpr\ht0-\dimen0\relax
    \dimen4=\dimexpr\dp0+\dimen0\relax
    \ifdim\dimen4>\dimen2 %
        \dimen2=\dimen4 %
    \fi
    \dimen4 = \dimen2 %
    \divide\dimen4 by 500 %
    \dimen4=\delimiterfactor\dimen4 %
    \dimen6=\dimexpr 2\dimen2 - \delimitershortfall\relax
    \ifdim\dimen6>\dimen4 %
        \dimen4=\dimen6 %
    \fi
    \sbox0{$#1\Bigl#3\Bigr#4$}%
    \ifdim\dimexpr\ht0+\dp0\relax>\dimen4 %
        \mleft#3#2\mright#4
    \else
        \Bigl#3#2\Bigr#4
    \fi
}
\makeatother

\makeatletter
\usepackage{faktor}
\newcommand{\quot}[2]{%
    \if@display%
        \faktor{#1}{#2}%
    \else%
        #1/#2%
    \fi%
}
\DeclareRobustCommand*{\mfaktor}[3][]%
{%
    {\mathpalette{\mfaktorImpl}{{#1}{#2}{#3}}}%
}
\newcommand*{\mfaktorImpl}[2]{\mfaktorImplImpl#1#2}
\newcommand*{\mfaktorImplImpl}[4]{
    \settoheight{\faktor@zaehlerhoehe}{\ensuremath{#1#2{#3}}}%
    \settoheight{\faktor@nennerhoehe}{\ensuremath{#1#2{#4}}}%
    \raisebox{-0.5\faktor@zaehlerhoehe}{\ensuremath{#1#2{#3}}}%
    \mkern-4mu\diagdown\mkern-5mu%
    \raisebox{0.5\faktor@nennerhoehe}{\ensuremath{#1#2{#4}}}%
}
\ifLuaTeX
    \newcommand{\lquot}[2]{%
            #2\backslash#1%
    }
\else
    \newcommand{\lquot}[2]{%
            #2\kern1pt\bs\kern0.5pt#1%
    }
\fi
\makeatother

\makeatletter
\let\oldoverline\overline
\def\getslant #1{\strip@pt\fontdimen1 #1}

\def\skoverline #1{\mathchoice
 {{\setbox\usefulbox=\hbox{$\m@th\displaystyle #1$}%
    \dimen@ \getslant\the\textfont\symletters \ht\usefulbox
    \divide\dimen@ \tw@
    \kern\dimen@
    \oldoverline{\kern-\dimen@ \box\usefulbox\kern\dimen@ }\kern-\dimen@ }}
 {{\setbox\usefulbox=\hbox{$\m@th\textstyle #1$}%
    \dimen@ \getslant\the\textfont\symletters \ht\usefulbox
    \divide\dimen@ \tw@
    \kern\dimen@
    \oldoverline{\kern-\dimen@ \box\usefulbox\kern\dimen@ }\kern-\dimen@ }}
 {{\setbox\usefulbox=\hbox{$\m@th\scriptstyle #1$}%
    \dimen@ \getslant\the\scriptfont\symletters \ht\usefulbox
    \divide\dimen@ \tw@
    \kern\dimen@
    \oldoverline{\kern-\dimen@ \box\usefulbox\kern\dimen@ }\kern-\dimen@ }}
 {{\setbox\usefulbox=\hbox{$\m@th\scriptscriptstyle #1$}%
    \dimen@ \getslant\the\scriptscriptfont\symletters \ht\usefulbox
    \divide\dimen@ \tw@
    \kern\dimen@
    \oldoverline{\kern-\dimen@ \box\usefulbox\kern\dimen@ }\kern-\dimen@ }}%
 {}}
\makeatother
\renewcommand{\overline}[1]{\skoverline{#1}}
\makeatletter

\newcommand{\DehnFunctionSymbol}{\delta}
\newcommand{\FillingFunctionSymbol}{\operatorname{FV}}
\newcommand{\FillingVolumeSymbol}{\operatorname{FVol}}
\newcommand{\VolumeSymbol}{\operatorname{Vol}}
\newcommand{\AreaSymbol}{\operatorname{Area}}

\NewDocumentCommand{\Dehn}{e{_^}}{%
    \DehnFunctionSymbol
    \IfValueT{#1}{_{#1}}%
    \IfValueT{#2}{^{#2}}%
}
\NewDocumentCommand{\Filling}{e{_^}}{%
    \FillingFunctionSymbol
    \IfValueT{#1}{_{\fillKern{#1}}}%
    \IfValueT{#2}{^{#2}}%
}
\NewDocumentCommand{\FillingVolume}{e{_^}}{%
    \FillingVolumeSymbol
    \IfValueT{#1}{_{\fillKern{#1}}}%
    \IfValueT{#2}{^{#2}}%
}
\NewDocumentCommand{\Volume}{e{_^}}{%
    \VolumeSymbol
    \IfValueT{#1}{_{\fillKern{#1}}}%
    \IfValueT{#2}{^{#2}}%
}
\NewDocumentCommand{\Vol}{e{_^}}{%
    \VolumeSymbol
    \IfValueT{#1}{_{\fillKern{#1}}}%
    \IfValueT{#2}{^{#2}}%
}
\NewDocumentCommand{\Mass}{m}{%
    \norm{#1}%
}
\NewDocumentCommand{\Area}{e{_^}}{%
    \AreaSymbol
    \IfValueT{#1}{_{\fillKern{#1}}}%
    \IfValueT{#2}{^{#2}}%
}

\hypersetup{%
    pdfauthor={Claudio Llosa Isenrich, Jannis Weis},%
    pdftitle={Dehn functions of groups acting on trees},%
    linktoc=page,
    hidelinks,
}

\usetikzlibrary{
    calc,
    intersections,
    decorations.pathmorphing, decorations.pathreplacing, decorations.shapes,
    trees,
    fit, backgrounds,
}

\input{figures/base.tex}

\newlist{steps}{enumerate}{1}
\setlist[steps]{label=\upshape\textbf{Step \arabic*},ref=\arabic*}
\zcRefTypeSetup{stepsi}{
    Name-sg = Step ,
    name-sg = step ,
    Name-pl = Steps ,
    name-pl = steps ,
    cap = false
}

\DeclareMathOperator{\Const}{const}
\DeclareMathOperator{\dist}{dist}
\DeclareMathOperator{\edgeDist}{edist}
\DeclareMathOperator{\MonDist}{Mon}
\newcommand{\Sol}{\operatorname{Sol}}

\newcommand{\collapse}{\mathrel{\searrow}}
\newcommand{\collapseFrom}{\mathrel{\swarrow}}

\newcommand{\transElem}[2]{h_{#1,#2}}
\newcommand{\mondElem}[3]{g_{#1,#2,#3}}
\newcommand{\mondGlue}[3]{\Psi_{#1,#2,#3}}
\newcommand{\transConj}[2]{c_{#1,#2}}
\newcommand{\transGlue}[2]{\psi_{#1,#2}}

\title[A Brown Theorem for Dehn functions of graphs of groups]
{A Brown Theorem for Dehn functions of \mbox{graphs of groups}}

\author{Claudio Llosa Isenrich}
\address{Department of Mathematics, University of Luxembourg, 6 Avenue de la Fonte, 4364 Esch-sur-Alzette, Luxembourg}
\email{claudio.llosaisenrich@uni.lu}

\author{Jannis Weis}
\address{\foreignlanguage{ngerman}{Fakultät für Mathematik, Karlsruher Institut für Technologie}, 76131 Karlsruhe, Germany}
\email{jannis.weis@kit.edu}


\date{August 7, 2026}

\begin{document}

\begin{abstract}
    \setlength{\parfillskip}{0pt plus 0.75\textwidth} 
    We prove an upper bound on the Dehn function of a group $G$ acting cellularly, cocompactly, and without inversions on a simply connected CW complex $X$ in terms of the Dehn functions of the vertex stabilizers, the Dehn function of $X$, and the distortion of the edge stabilizers, provided that $X$ is either a tree or the stabilizer of each $2$-cell has finite index in the stabilizer of every edge in its boundary. This provides a Dehn function analogue of Brown's Theorem for finiteness properties and an answer to a question of Zaremsky in these cases. We also prove analogues of our result for higher Dehn functions when $X$ is a tree.
\end{abstract}

\maketitle

\section{Introduction}
The Dehn function $\Dehn_G(n)$ of a finitely presented group $G$ provides a quantitative measure of the worst-case complexity of a brute-force approach to the word problem for words of length $\leq n$. It can be interpreted geometrically as an optimal isoperimetric function $\Dehn_Y$ for null-homotopic loops of length $\leq n$ in the universal cover $\widetilde{Y}$ of a closed manifold or compact CW complex $Y$ with fundamental group $\pi_1(Y)\cong G$. This geometric interpretation readily generalizes to the higher-dimensional Dehn functions $\Dehn_Y^k(n)$, which are defined as optimal isoperimetric functions for null-homotopic $k$-spheres of volume $\leq n$ in $\widetilde{Y}$ under the additional assumption that the universal cover $\widetilde{Y}$ is $k$-connected. Under these assumptions, $\Dehn_Y^k$ depends only on the fundamental group $G=\pi_1(Y)$, up to asymptotic equivalence, and is a quasi-isometry invariant among groups of finiteness type $\F_{k+1}$~\cite{AlonsoWangPride-1999}; here we say that a group is of type $\F_k$ if it admits a classifying space with finite $k$-skeleton. We thus often write $\Dehn_G^k$ instead of $\Dehn_Y^k$, with the understanding that this denotes an equivalence class of functions.

Taking the point of view of finiteness properties and recalling that type $\F_n$ is itself a quasi-isometry invariant~\cite{Alonso-1994}, one can view $k$-dimensional Dehn functions $\Dehn_G^k$ as quasi-isometry invariants that further quantify the finiteness properties $\F_{k+1}$. This naturally raises the question of the extent to which results on finiteness properties admit quantitative analogues in terms of Dehn functions. A classical result of this kind is Brown's Theorem.
\begin{thmalpha}[{Brown~\cite{Brown-1987}}]
    Let $G$ be a group that acts cellularly and cocompactly on a $(k-1)$-connected CW complex $X$. Assume that the cell stabilizer $G_\sigma$ of every $i$-cell $\sigma$ is $\F_{k-i}$ for every $0\leq i \leq k$. Then $G$ is $\F_k$.
\end{thmalpha}

The purpose of this work is to explore the existence of analogues of Brown's criterion for Dehn functions. A version of the question whether such an analogue exists for the first Dehn function $\Dehn_G^1$ was raised by Zaremsky~\cite{Zaremsky-open-problems}*{Problem 1.17}. We give an affirmative answer if $X$ is $1$-dimensional. To state it, we recall that a relative $0$-dimensional analogue of the Dehn function is the distortion $\dist_H^G$ of a subgroup $H\leq G$. We denote by $\overline{f}$ the superadditive closure of a function $f$.

\begin{thmalpha}%
    \label{thm:Dehn-Brown-trees}
    Let $G$ be a group that acts cellularly and cocompactly without inversion on a tree $T$. Assume that all vertex stabilizers $G_v$ are finitely presented and that all edge stabilizers $G_e$ are finitely generated. Then
    \[
        \Dehn_G(n)\preccurlyeq n\cdot \max_{\mathclap{v\in V(T)}}\; \Dehn_{G_v} \left(\overline{\edgeDist_T}(n)\right),
    \]
    where $\edgeDist_T(n)=\max_{e\in E(T)} \dist_{G_e}^G(n)$.
\end{thmalpha}

The assumptions in \cref{thm:Dehn-Brown-trees} are equivalent to saying that $G$ is the fundamental group of a finite graph of groups, where, up to conjugation, the $G_v$ and $G_e$ are its vertex and edge groups. In particular, the maxima in the statement exist. Similar considerations imply that the maxima in all subsequent statements exist. The form of the estimate in \cref{thm:Dehn-Brown-trees} is sharp in general. Indeed, the Baumslag--Solitar group $BS(1,2)=\left\langle a,b\mid bab^{-1}=a^2\right\rangle$ decomposes as an HNN extension with exponentially distorted infinite cyclic edge and vertex stabilizers. Hence, \cref{thm:Dehn-Brown-trees} yields an exponential upper bound on $\Dehn_{BS(1,2)}$, which is known to be optimal~\cite{Gersten-1992}.

While \cref{thm:Dehn-Brown-trees} does not seem to appear in the literature for arbitrary graphs of groups, previous results cover several special cases. Brick studied the behavior of Dehn functions under amalgamated products and HNN extensions, while Brick and Corson obtained sharper estimates for amalgamated products along strongly undistorted subgroups~\cites{Brick-1993,Brick-Corson-2000}. Brick and Corson also proved a general estimate for finite developable complexes of groups with finite edge groups in terms of the Dehn functions of the vertex groups and the Howie function of the complex~\cite{brick+corson}. Estimates for doubles $G \ast_H G$ and HNN extensions $G\ast_H$ of a finitely presented group $G$ with respect to a finitely generated subgroup $H\leq G$ appear in~\cite{bridson+haefliger}*{Thm. III.$\Gamma$.6.20}. Doubles and certain iterated ascending HNN extensions have played an important role in gaining new insights into Dehn functions, and they admit generalizations to higher dimensions and to homological Dehn functions~\cites{BBFS-2009,LiManin-2022}.

We generalize our results to Dehn functions of groups acting on higher-dimensional simplicial complexes and to higher Dehn functions of groups acting on trees.
\begin{thmalpha}%
    \label{thm:Dehn-Brown-complexes}
    Let $G$ be a group that acts cellularly and cocompactly without inversion on a simply connected simplicial complex $X$. Assume that all vertex stabilizers $G_v$ are finitely presented, that all edge stabilizers $G_e$ are finitely generated, and that $[G_e : G_\sigma]< \infty$ for all $2$-cells $\sigma$ and edges $e\subset \sigma$. Then
    \[
        \Dehn_G(n)\preccurlyeq \Dehn_X(n)\cdot \max_{v\in X^{(0)}} \Dehn_{G_v} \left(\overline{\edgeDist_X}(\Dehn_X(n))\right),
    \]
    where $\edgeDist_X(n)=\max_e \dist_{G_e}^G(n)$ and $e$ runs over all edges of $X$.
\end{thmalpha}

For higher Dehn functions, we obtain the following result, where $\Filling_H^{k+1}$ denotes the $(k+1)$-dimensional integral homological filling function of $H$; see \cref{sec:Background} for a definition and note that, for historical reasons, there is an index shift.
\begin{thmalpha}%
    \label{thm:high-dimensional-Brown-trees}
    For $k\geq 2$, let $G$ be a group that acts cellularly and cocompactly without inversion on a tree $T$. Assume that all vertex stabilizers $G_v$ are $\F_{k+1}$ and that all edge stabilizers $G_e$ are $\F_k$. Then
    \[
        \Filling_G^{k+1}(n)\preccurlyeq \overline{f}(\overline{g}(n)),
    \]
    where $f(n)=\max_{v\in V(T)} \Filling_{G_v}^{k+1}(n)$ and $g(n)=\max_{e\in E(T)} \Filling_{G_e}^{k}(n)$.
\end{thmalpha}
Note that Barnard, Brady, and Dani~\cite{BarnardBradyDani-2012}*{Proposition 6.1} proved an analogue of \cref{thm:high-dimensional-Brown-trees} for the second-order homotopical Dehn function $\Dehn_G^2$. Moreover, in very recent work, Sauer and the second author proved a version of \cref{thm:high-dimensional-Brown-trees} with a polynomial upper bound for complexes of groups under the assumption that the filling functions of the complex, the homological Dehn functions of all cell stabilizers, and the distortion of cell stabilizers are polynomially bounded~\cite{SauerWeis-2026}.

It is known that $\Dehn_H^k$ and $\Filling_H^{k+1}$ need not be equivalent for $k=1, 2$, while for $k\geq 3$ we have $\Dehn_H^k\approx \Filling_H^{k+1}$~\cites{AbramsBradyDaniYoung-2013,Young-2011}. Thus, we deduce the following higher-dimensional version of~\cite{BarnardBradyDani-2012}*{Proposition 6.1} from \cref{thm:high-dimensional-Brown-trees}.
\begin{coralpha}%
    \label{cor:high-dimensional-Brown-trees}
    For $k\geq 4$, let $G$ be a group that acts cellularly and cocompactly without inversion on a tree $T$. Assume that all vertex stabilizers $G_v$ are $\F_{k+1}$ and that all edge stabilizers $G_e$ are $\F_k$. Then
    \[
        \Dehn_G^k(n)\preccurlyeq \overline{f}(\overline{g}(n)),
    \]
    where $f(n)=\max_{v\in V(T)} \Dehn_{G_v}^k(n)$ and $g(n)=\max_{e\in E(T)} \Dehn_{G_e}^{k-1}(n)$.
\end{coralpha}

Note that, for $k=2$, we have $\Dehn_H^2\preccurlyeq \Filling_H^{3}$. Thus, we cannot deduce \cref{cor:high-dimensional-Brown-trees} in the stated form for $k=3$. However, we can still bound $\Dehn_G^3$ analogously in terms of the Dehn functions $\Dehn_{G_v}^3$ and $\Filling_{G_e}^3$ for $v\in V(T)$ and $e\in E(T)$.

\subsection*{Structure}
In \cref{sec:Background} we introduce all required background on distortion, Dehn functions, collapses, and complexes of groups. In \cref{sec:dimension-one} we prove \cref{thm:Dehn-Brown-trees} and \cref{thm:Dehn-Brown-complexes}. In \cref{sec:high-dimension} we prove \cref{thm:high-dimensional-Brown-trees}. For completeness, we also include a proof of the analogue of \cref{cor:high-dimensional-Brown-trees} for $k=2$, although this is~\cite{BarnardBradyDani-2012}*{Proposition 6.1}. Finally, in \cref{sec:applications} we discuss the special case of group extensions as well as an application to amalgamated products of nilpotent groups over cyclic groups.

\subsection*{Acknowledgements}
The authors thank Uri Bader, Jer\'onimo Garc\'ia-Mej\'ia and Roman Sauer for helpful discussions.

\section{Preliminaries}\label{sec:Background}
In this section, we introduce all relevant background and notions concerning distortion, Dehn functions, collapses, and classifying spaces for complexes of groups.

\subsection{Distortion, van~Kampen diagrams and Dehn functions}

We will give a brief introduction to distortion, van~Kampen diagrams and Dehn functions. For more detailed accounts we refer to~\cites{Bridson-2002-a,BradyRileyShort-2007} in the case of $1$-dimensional homotopical Dehn functions and to~\cites{AlonsoWangPride-1999,Bridson-2002-b,BBFS-2009,Young-2011} for higher-dimensional homotopical and homological Dehn functions.

The notions of distortion and Dehn functions for groups are well-defined only up to certain equivalence relations on functions $f, g\colon \mathbb{N}_0\to \mathbb{R}_{\geq 0}$. We write $f\lesssim g$ (resp. $f\preccurlyeq g$) if there is a constant $C>0$ such that $f(n)\leq Cg(Cn+C)+C$ (resp. $f(n)\leq Cg(Cn+C)+Cn+C$). We denote by $\simeq$ (resp. $\approx$) the equivalence relation induced by $\lesssim$ (resp. $\preccurlyeq$). We say that a function $f\colon \mathbb{N}_0\to \mathbb{R}_{\geq 0}$ is \emph{superadditive} if, for every partition $n=n_1+\cdots+n_k$ of $n\in \mathbb{N}_0$ into nonnegative integers, we have $f(n)\geq f(n_1)+\cdots+f(n_k)$. The \emph{superadditive closure} of a function $f\colon \mathbb{N}_0\to \mathbb{R}_{\geq 0}$ is the smallest superadditive function $\overline{f}\colon \mathbb{N}_0\to \mathbb{R}_{\geq 0}$ with $f\leq \overline{f}$; it can be defined by $\overline{f}(n)\defq\max_{n=n_1+\cdots+n_k} \sum_{i=1}^k f(n_i)$, where the maximum ranges over partitions with $n,n_i\geq 1$.

For a CW complex $X$, we denote by $X^{(n)}$ its $n$-skeleton and by $d_X\colon X^{(0)}\times X^{(0)}\to \mathbb{N}_{0}$ the combinatorial metric on its $1$-skeleton, where every edge has length $1$. A \emph{combinatorial path} $\gamma$ in the $1$-skeleton $X^{(1)}$ is a path defined by a consecutive choice of edges. Its length $\ell_X(\gamma)$ is its number of edges.

For a finitely generated group $G=\langle S\rangle$, we denote by $d_{G,S}$ its word metric, which is the combinatorial metric on its Cayley graph, and by $\ell_{G,S}(g)\defq d_{G,S}(g,1)$ the word length of an element $g\in G$. When a generating set is understood, we usually write simply $d_G$ and $\ell_G$.

\begin{defn}
    Let $G=\langle S\rangle$ and $H=\langle T\rangle$ be finitely generated groups with $H\leq G$. The \emph{distortion} $\dist_H^G\colon \mathbb{N}_0\to \mathbb{R}_{\geq 0}$ of $H$ in $G$ is
    \[
        \dist_H^G(n)\defq\max\left\{\ell_H(h) \mid h\in H \mbox{ with } \ell_G(h)\leq n \right\}.
    \]
\end{defn}

The distortion is independent of the choice of finite generating set up to the $\simeq$-equivalence.

We will now define higher-dimensional Dehn functions, following the definitions in~\cites{Bridson-2002-b,Riley-2003}. For further references with equivalent definitions, see~\cites{AlonsoWangPride-1999,BBFS-2009,Young-2011}.

We start by recalling that a \emph{combinatorial map} between CW complexes is a map that maps open $n$-cells homeomorphically onto open $n$-cells. A \emph{combinatorial complex} $K$ is defined iteratively as a CW complex such that the attaching maps $\phi_\sigma \colon \partial D^{n+1}\to K^{(n)}$ of $(n+1)$-cells are combinatorial with respect to some combinatorial structure on $S^n=\partial D^{n+1}$. A \emph{singular combinatorial map} $f\colon L\to K$ between CW complexes is a continuous map such that, for every open $n$-cell $e$ in $L$, either $f|_e$ is a homeomorphism onto an open $n$-cell of $K$ or $f(e)\subseteq K^{(n-1)}$; in the latter case, we say that $e$ is \emph{singular}. A \emph{singular combinatorial complex} $K$ is a CW complex such that the attaching maps $\phi_\sigma \colon \partial D^{n+1}\to K^{(n)}$ are singular combinatorial maps with respect to some combinatorial structure on $S^n=\partial D^{n+1}$.

\begin{defn}%
    \label{def:admissible-maps}
    Let $k\geq 1$, let $X$ be a singular combinatorial complex, and let $W$ be a compact $k$-dimensional manifold. An \emph{admissible map} $f\colon W\to X$ is a continuous map together with a combinatorial cell structure on $W$ such that $f$ is a singular combinatorial map. We define the volume $\Volume_X(f)$ to be the number of open $k$-cells of $W$ on which the restriction of $f$ is a homeomorphism onto an open $k$-cell of $X$.
\end{defn}

\begin{defn}
    Let $X$ be a $k$-connected singular combinatorial complex and let $\alpha\colon S^k\to X$ be an admissible map. Its filling volume is
    \[
        \FillingVolume_X(\alpha)\defq\min \left\{\Volume_X(\beta)\mid \beta\colon D^{k+1}\to X \mbox{ admissible with } \beta|_{\partial D^{k+1}}=\alpha \right\}.
    \]
    The \emph{$k$-dimensional (homotopical) Dehn function} of $X$ is
    \[
        \Dehn_X^k(n)\defq\max \left\{\FillingVolume_X(\alpha) \mid \Volume_X(\alpha)\leq n, ~\alpha: S^k \to X \mbox{ admissible}\right\}.
    \]
    The \emph{$k$-dimensional (homotopical) Dehn function} of a group $G$ of type $\F_{k+1}$ is $\Dehn_G^k(n)\defq\Dehn_X^k(n)$, where $X$ is any $k$-connected singular combinatorial complex that admits a free, cocompact, cellular $G$-action.
\end{defn}

\begin{rem}
    The definition of $\Dehn_X^k$ in~\cites{Bridson-2002-b,Riley-2003} uses based spheres and a basepoint in $X$. Since we consider only complexes with a cocompact group action, our basepoint-independent definition is equivalent to their definition up to $\approx$-equivalence of functions.

    For every group $G$ of type $\F_{k+1}$, $k\geq 1$, there is a $k$-connected singular combinatorial complex that admits a free, cocompact, cellular $G$-action. It can be constructed as the universal cover of the $(k+1)$-skeleton of a classifying space for $G$. However, it is not unique in general, and the definition of $\Dehn_G^k$ depends on its choice. Thus, formally, $\Dehn_G^k$ is defined only up to $\approx$-equivalence. More generally, its $\approx$-equivalence class is a quasi-isometry invariant among groups of type $\F_{k+1}$.
\end{rem}

In the special case $k=1$, one can also work with combinatorial maps by considering compact planar combinatorial $2$-complexes instead of disks. These are known as van~Kampen diagrams. We recall the main facts and definitions that we will use and refer to~\cites{Bridson-2002-a,BradyRileyShort-2007} for further details.
\begin{defn}
    Let $X$ be a combinatorial $2$-complex. A \emph{van~Kampen diagram} $(\Delta,f,\iota)$ for a null-homotopic combinatorial loop $\gamma\colon S^1 \to X^{(1)}$ is a compact contractible planar combinatorial $2$-complex $\Delta$ together with a combinatorial map $f\colon \Delta\to X^{(2)}$ such that $f|_{\partial \Delta}\circ \iota=\gamma$, where $\iota\colon S^1 \to \partial \Delta$ is a combinatorial parametrization of $\partial \Delta$. We denote by $\Area(f)\defq\Vol(f)$ the number of $2$-cells of $\Delta$.
\end{defn}

\begin{rem}
    If $G$ is a group that acts freely, cellularly, and cocompactly on a combinatorial $2$-complex $X$, then replacing admissible maps by van~Kampen diagrams in the definition of the Dehn function of $X$ yields a function that is $\approx$-equivalent to $\Dehn_G^1$. We can thus restrict to van~Kampen diagrams when working with the one-dimensional Dehn function $\Dehn_G\defq\Dehn_G^1$.
\end{rem}

\begin{rem}
    For every finitely presented group $G\cong \mathcal{P}=\left\langle S \mid R \right\rangle$, the universal cover $\widetilde{X}_\mathcal{P}$ of its presentation complex is a combinatorial $2$-complex. It can be equipped with a labeling of its edges by generators and its $2$-cells by relations. For a van~Kampen diagram $f\colon \Delta \to X^{(2)}$, this induces a labeling of its edges and $2$-cells by generators and relations. Van Kampen's Lemma then provides a translation between van~Kampen diagrams for combinatorial loops and free decompositions of null-homotopic words in $S$ as products of conjugates of relations from $R$. This provides a translation between isoperimetric functions and the complexity of a brute-force approach to the word problem.
\end{rem}

There is a homological version of $k$-dimensional Dehn functions for groups $G$ of type $\F_{k+1}$, defined in terms of cycles $\alpha$ and chains $\beta$ with $\partial \beta=\alpha$ in the chain complex of the universal cover of a classifying space with finite $(k+1)$-skeleton. In fact, the weaker homological finiteness property $\FP_{k+1}$, defined in terms of finitely generated resolutions by projective $\mathbb{Z}G$-modules, would suffice to define it. We do not discuss these finiteness properties further because we do not require them. As with homotopical Dehn functions, we define homological Dehn functions for more general complexes. We refer to~\cite{Young-2011} for further details.

For a CW complex $X$, we equip its cellular chain complex with integral coefficients $C_i(X,\mathbb{Z})$ with the $1$-norm $\Mass{\alpha}\defq \sum_{i\in I} |a_i|$ for $\alpha=\sum_{i\in I} a_i \sigma_i\in C_i(X,\mathbb{Z})$.
\begin{defn}
    Let $X$ be a $k$-connected CW complex. We define the filling volume of a $k$-cycle $\alpha\in Z_k(X,\mathbb{Z})$ as
    \[
        \FillingVolume_X^{k+1}(\alpha)\defq \inf \left\{ \Mass{\beta} \mid \beta \in C_{k+1}(X,\mathbb{Z}),~\partial\beta = \alpha\right\}.
    \]
    The \emph{$(k+1)$-dimensional homological Dehn function of $X$} (with integer coefficients) is
    \[
        \Filling_X^{k+1}(n)\defq\sup\left\{\FillingVolume_X^{k+1}(\alpha)\mid \alpha \in Z_k(X,\mathbb{Z}),~\Mass{\alpha}\leq n\right\}.
    \]
    For a group $G$ of type $\F_{k+1}$, let $X$ be a $k$-connected CW complex with a free cellular cocompact $G$-action. The \emph{$(k+1)$-dimensional homological Dehn function of $G$} (with integer coefficients) is $\Filling_G^{k+1}(n)\defq\Filling_X^{k+1}(n)$.
\end{defn}

Note that, for historical reasons, there is an index shift by $+1$ between homotopical and homological Dehn functions. As for homotopical Dehn functions, the $(k+1)$-dimensional homological Dehn function of a group $G$ of type $\F_{k+1}$ is well-defined up to $\approx$-equivalence~\cite{Young-2011}*{Lemma 1}. One can also define homological Dehn functions with other normed coefficient rings. Here we restrict to integral coefficients for simplicity.

\subsection{Collapsible complexes}

We recall the terminology of collapses and collapsible complexes, following Whitehead~\cite{whitehead}.

\begin{defn}
    Let $K$ be a simplicial complex. A simplex $\tau \in K$ is called a \emph{free face} of a simplex $\sigma \in K$ if $\tau$ is a proper face of $\sigma$ and $\sigma$ is the unique maximal simplex in $K$
    containing $\tau$.

    Removing all intermediate simplices between $\tau$ and $\sigma$ from $K$, that is, removing the set $\{\tau' \mid \tau \subseteq \tau' \subseteq \sigma \}$ yields a subcomplex $L$.
    This operation is called a \emph{collapse}, and we write $K \collapse L$.
    If $\tau$ has codimension $1$ in $\sigma$, it is called an \emph{elementary collapse}.

    The complex $K$ is called \emph{collapsible} if there exists a sequence of elementary collapses collapsing $K$ to a point.
\end{defn}

\begin{lem}[\cite{whitehead}]%
    \label{lem:collapse-in-order}
    Let $K$ be a finite collapsible simplicial complex of dimension $n$.
    Then there exists a sequence of collapses
    \[
        K = K_0 \collapse K_1 \collapse \ldots \collapse K_r = \{\ast\}\text{,}
    \]
    such that, for every $1 \leq i \leq n - 1$, every collapse of an $(i+1)$-cell precedes every collapse of an $i$-cell.
\end{lem}
\begin{proof}
    Collapsing an $i$-cell cannot create new free faces of $(i+1)$-cells; hence, if a collapse of an $i$-cell
    occurs after a collapse of an $(i+1)$-cell, they are independent and can also be performed in the opposite order.
\end{proof}

\begin{example}
    There is only one type of collapse of a $1$-cell, namely the elementary collapse of the cell along one of its vertices.
    There are two types of collapses of a $2$-cell: collapse along a free vertex and collapse along a free edge.
    The non-elementary collapse along a free vertex is equivalent to first collapsing an adjacent free edge and then collapsing the remaining edge through an elementary collapse of a $1$-cell (see \cref{fig:collapse-types}).
\end{example}

\begin{lem}%
    \label{lem:planaer-complex-collapsible}
    Every finite contractible planar simplicial $2$-complex $K$ is collapsible.
\end{lem}
\begin{proof}
    If $K$ is not a tree, then it has a $1$-simplex $e$ on its boundary that is a face of a $2$-simplex $\sigma$.
    Performing an elementary collapse of $\sigma$ along $e$ reduces the number of $2$-simplices of the complex by one.
    Once all $2$-simplices have been collapsed, the remaining complex is $1$-dimensional and still contractible, hence a tree, which is clearly collapsible.
\end{proof}

\begin{figure}
    \centering
    \begin{tikzpicture}[
        line width=0.8pt,
        deleted/.style={red},
    ]
    \newcommand{\Vertex}[2][]{
        \fill[#1] (#2) circle[radius=1.5pt];
    }
    \newcommand{\TriangleBase}[1][]{%
        \filldraw[line width=0.8pt,#1] (0,0) coordinate (A) -- (1.2,0) coordinate (B) -- (0.6,1.05) coordinate (C) -- cycle;
    }
    \newcommand{\Triangle}[2][]{%
        \TriangleBase[#1]
        #2
        \foreach \P in {A,B,C}{\Vertex{\P}}
    }
    \newcommand{\Collapse}[1][0.55]{
        \draw[draw=none] (1.55,#1) -- node {\Large$\collapse$} (2.35,#1);
    }
    \newcommand{\CollapseFrom}[1][0.55]{
        \draw[draw=none] (2.35,#1) -- node {\Large$\collapseFrom$} (1.55,#1);
    }

    \begin{scope}[shift={(0,0)}]
        \begin{scope}
            \draw[deleted] (0,0) coordinate (A) -- (1.2,0) coordinate (B);
            \Vertex[deleted]{A}
            \Vertex{B}
        \end{scope}
        \Collapse[0]
        \begin{scope}[shift={(2.7,0)}]
            \Vertex{0,0}
        \end{scope}
    \end{scope}

    \begin{scope}[shift={(0,-2.45)}]
        \Triangle[fill=red,fill opacity=0.25]{
            \draw[deleted] (A) -- (B);
        }
        \Collapse
        \begin{scope}[shift={(2.7,0)}]
            \Triangle[opacity=0]{
                \draw (A) -- (C);
                \draw (B) -- (C);
            }
        \end{scope}
    \end{scope}

    \begin{scope}[shift={(0,-4.9)}]
        \Triangle[fill=red,fill opacity=0.25]{}
        \Vertex[deleted]{C}
        \Collapse
        \begin{scope}[shift={(2.7,0)}]
            \draw (0,0.55) coordinate (A) -- (1.2,0.55) coordinate (B);
            \Vertex{B}
            \Vertex{A}

            \CollapseFrom
        \end{scope}
        \begin{scope}[shift={(5.4,0)}]
            \Triangle[opacity=0]{
                \draw (A) -- (B);
                \draw[deleted] (B) -- (C);
            }
            \Vertex[deleted]{C}
            \CollapseFrom
        \end{scope}
        \begin{scope}[shift={(8.1,0)}]
            \Triangle[fill=red,fill opacity=0.25]{
                \draw[deleted] (A) -- (C);
            }
        \end{scope}
    \end{scope}
\end{tikzpicture}
    \caption{%
        \label{fig:collapse-types}
        Types of collapses of a $1$-cell and a $2$-cell. The second type of $2$-cell collapse is the composition
        of two elementary collapses.
    }
\end{figure}
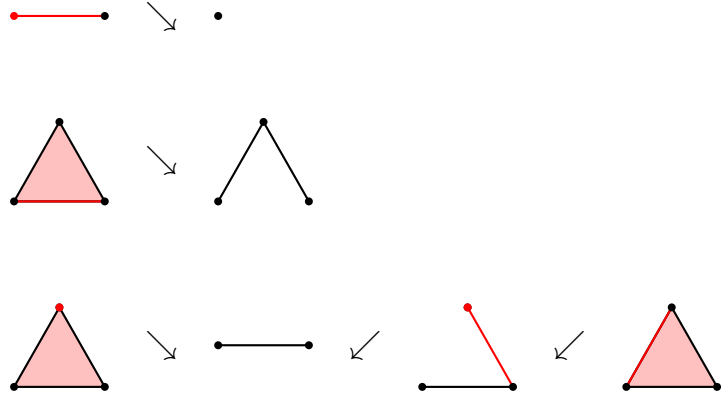

\subsection{\texorpdfstring{$G$}{G}-complexes and the Haefliger construction}

To prove our main results, we use the following construction of classifying spaces for complexes of groups due to Haefliger~\cite{haefliger}; see also~\cite{bridson+haefliger}*{III.$\mathcal{C}.2$}. We refer to~\cite{drutu+kapovich} for a more detailed account of the construction. We follow the conventions in~\cite{bridson+haefliger}*{III.$\mathcal{C}$.2} and use similar notation.

\begin{thm}[Haefliger construction~\cite{haefliger}]%
    \label{thm:haefliger-construction}
    Let $X$ be a cocompact $G$\=/simplicial complex. Then there exists a free regular $G$-CW complex $\widehat{X}$
    together with a $G$\=/equivariant map $\pi \colon \widehat{X} \to X$ and a filtration $\widehat{X} = \colim_n \widehat{X}^{[n]}$ with $\widehat{X}^{[-1]} = \emptyset$ such that
    \begin{enumerate}
        \item $\pi$ restricts to a homotopy equivalence $\widehat{X}^{[n]} \to X^{(n)}$;
        \item $\widehat{X}^{[n]}$ fits into a pushout diagram of the form
              \[\begin{tikzcd}
                  {\coprod_{\sigma \in I_n} G \times_{G_\sigma} EG_\sigma \times \partial \sigma} & {\widehat{X}^{[n-1]}} \\
                  {\coprod_{\sigma \in I_n} G \times_{G_\sigma} EG_\sigma \times \sigma} & {\widehat{X}^{[n]}.}
                  \arrow[from=1-1, to=1-2]
                  \arrow[hook, from=1-1, to=2-1]
                  \arrow[from=1-2, to=2-2]
                  \arrow[from=2-1, to=2-2]
              \end{tikzcd}\]
    \end{enumerate}
\end{thm}

\begin{construction}%
    \label{rem:description-haefliger-construction}
    Let $Y = \lquot{X}{G}$ and $p \colon X \to Y$ be the projection map.
    For every simplex $\sigma$ of $Y$, choose a simplex $\overline{\sigma}$ in $X$ such that $p(\overline{\sigma}) = \sigma$.
    Let $\tau \subset \sigma$. Because $G$ acts without inversion, there exists a unique
    face $\tau_\sigma$ of $\overline{\sigma}$ such that $p(\tau_\sigma) = \tau$.
    Then in general $\tau_\sigma \neq \overline{\tau}$.
    However, $\tau_\sigma$ and $\overline{\tau}$ lie in the same $G$-orbit, so we can choose an element
    $\transElem{\tau}{\sigma} \in G$ such that $\transElem{\tau}{\sigma} \tau_\sigma = \overline{\tau}$.
    For each $\sigma \in Y$, let $G_\sigma$ be the stabilizer of $\overline{\sigma}$ and define
    \[
        \transConj{\tau}{\sigma} \colon G_\sigma \to G_\tau, \quad g \mapsto \transElem{\tau}{\sigma} g \transElem{\tau}{\sigma}^{-1}
    \]
    for every face inclusion $\tau \subseteq \sigma$.
    The fiber $\pi^{-1}(\sigma)$ of a cell $\sigma \in X$ is the copy of $EG_{p(\sigma)} \times \overline{p(\sigma)}$
    in
    \[
        G \times_{G_{p(\sigma)}} EG_{p(\sigma)} \times \overline{p(\sigma)}
        = \coprod_{\mathclap{G/G_{p(\sigma)}}} EG_{p(\sigma)} \times \overline{p(\sigma)}
    \]
    indexed by the coset $g G_{p(\sigma)}$ satisfying $g \cdot \overline{p(\sigma)} = \sigma$.
    It is sometimes necessary to distinguish how these copies of $EG_{p(\sigma)}$ sit inside $\widehat{X}$.
    In this case we write $\widehat{X}_\sigma$ for the copy such that $\pi^{-1}(\sigma) = \{gG_\sigma\} \times \widehat{X}_\sigma \times \overline{p(\sigma)}$.

    In particular, for an edge $e = \{u,v\} \in \lquot{X}{G}$, let $u_e$ and $v_e$ be the faces of $\overline{e}$
    with $p(u_e) = u$ and $p(v_e) = v$, respectively.
    Since $\transElem{u}{e} u_e = \overline{u}$ by construction, the vertex of the edge $g \cdot \overline{e}$
    lying over $u$ is $g u_e = g \transElem{u}{e}^{-1} \overline{u}$.
    The attaching map of the edge fiber to the vertex fiber at $u$ is therefore the $G$-equivariant map
    \begin{align*}
        \transGlue{u}{e} \colon G \times_{G_{e}} EG_{e} \times \{u_e\} &\longrightarrow G \times_{G_{u}} EG_{u} \\
        [g, x]                                                         &\longmapsto [g \transElem{u}{e}^{-1}, \transGlue{u}{e}(x)],
    \end{align*}
    where $\transGlue{u}{e} \colon EG_{e} \to EG_{u}$ is the map induced by $\transConj{u}{e}$,
    and analogously for the attaching map to the fiber of $v$.
    More generally, if $\tau \subseteq \sigma$ and $\tau_\sigma\subseteq \overline{\sigma}$ projecting to $\tau$, we consider the map
    \begin{align*}
        \Phi_{\tau,\sigma}\colon
        G \times_{G_\sigma} EG_\sigma \times \tau_\sigma
                &\longrightarrow G \times_{G_\tau} EG_\tau \times \overline{\tau} \\
        [g,x,y] &\longmapsto [g\transElem{\tau}{\sigma}^{-1},\transGlue{\tau}{\sigma}(x),\transElem{\tau}{\sigma}y],
    \end{align*}
    where again $\transGlue{\tau}{\sigma}\colon EG_\sigma \to EG_\tau$ is induced by
    $\transConj{\tau}{\sigma}$.
    This is not the attaching map itself because, for iterated face inclusions $\rho \subseteq \tau \subseteq \sigma$, we cannot in general assume that
    $\transGlue{\rho}{\sigma} = \transGlue{\rho}{\tau} \circ \transGlue{\tau}{\sigma}$.
    However, they always differ by homotopies
    \[
        \mondGlue{\rho}{\tau}{\sigma} \colon EG_\sigma \times [0,1] \to EG_\rho
    \]
    depending on the monodromy element
    \[
        \mondElem{\rho}{\tau}{\sigma}
        \defq
        \transElem{\rho}{\tau}
        \transElem{\tau}{\sigma}
        \transElem{\rho}{\sigma}^{-1}.
    \]
    The explicit construction of the correct attaching maps is more involved; we refer to~\cite{drutu+kapovich}*{Section 5.8.2} for details.
\end{construction}

For our purposes, we require the CW structure on $\widehat{X}$ and the map $\pi$ to be well behaved up to a certain dimension. This can be achieved by performing a more careful version of Haefliger's construction.

\begin{rem}%
    \label{rem:polyhderal-model-1-skeleton}
    Choose finite generating sets $S_e$ for the edge stabilizers and $S_v$ for the
    vertex stabilizers in $\lquot{X}{G}$.
    After replacing $S_v$ by $S_v \cup \bigcup_{v\subset e} \transConj{v}{e}(S_e)$,
    we may assume that $\transConj{v}{e}(S_e)\subseteq S_v$ for every incidence
    $v\subset e$.
    Choose the models $EG_e$ and $EG_v$ such that their $1$-skeleta are the
    corresponding Cayley graphs.
    Then $\transGlue{v}{e}\colon EG_e^{(1)}\to EG_v^{(1)}$ sends each edge labeled
    by $s\in S_e$ to the edge labeled by $\transConj{v}{e}(s)$, and is therefore
    injective. This injectivity will be used repeatedly in our proofs.
\end{rem}

The setup in \cref{rem:polyhderal-model-1-skeleton} is sufficient for the proof of \cref{thm:Dehn-Brown-trees} in \cref{sec:dimension-one} after observing that one can choose the cells over edges of $X$ to be product cells (see \cref{rem:polyhedral-model-1-skeleton-refined}).
In higher dimensions we will require $\widehat{X}$ and $\pi$ to be polyhedral up to a given dimension and more work is required to guarantee this.
We start by introducing the terminology of polyhedral complexes and maps.

\begin{defn}[Polyhedral complex]
    A \emph{polyhedral complex} is a combinatorial complex in which every $n$-cell is equipped with the structure of an $n$-dimensional compact convex polytope and the attaching maps are affine on open cells.
\end{defn}

\begin{defn}[Polyhedral map]%
    A map $f \colon X \to Y$ between polyhedral complexes is \emph{polyhedral} if, for every
    cell $\sigma$ of $X$, there is a cell $\tau$ of $Y$ such that
    $f(\sigma^\circ)=\tau^\circ$ and the restriction $f|_\sigma\colon \sigma\to\tau$ is an affine surjection.
\end{defn}

\begin{defn}[$G$-polyhedral complex]%
    \label{defn:G-polyhedral-complex}
    Let $G$ be a group. A \emph{$G$-polyhedral complex} is a $G$-space $X$ that admits a filtration $\emptyset = X^{(-1)} \subseteq X^{(0)} \subseteq X^{(1)} \subseteq \cdots$
    with $X = \colim_{n} X^{(n)}$, such that each $X^{(n)}$ fits into a polyhedral equivariant pushout diagram of the form
    \[\begin{tikzcd}
        {\coprod_{i \in I_n} G/G_i \times \partial P_i} & {X^{(n-1)}} \\
        {\coprod_{i \in I_n} G/G_i \times P_i} & {X^{(n)},}
        \arrow["\alpha_n", from=1-1, to=1-2]
        \arrow[hook, from=1-1, to=2-1]
        \arrow[from=1-2, to=2-2]
        \arrow["e_n",from=2-1, to=2-2]
    \end{tikzcd}\]
    where
    \begin{condenum}
        \item $I_n$ is an index set;
        \item each $P_i$ is a compact convex $n$-polytope;
        \item each $G_i \leq G$ is a subgroup called the \emph{cell stabilizer} of the cell represented by $i$.
    \end{condenum}
    The complex $X$ is \emph{cocompact} if all index sets $I_n$ are finite and only finitely many are non-empty.
    It is called \emph{regular} if every attaching map $e_n$ is a homeomorphism onto its image.
\end{defn}

\begin{rem}
    Equivalently, a $G$-polyhedral complex is a polyhedral complex $X$ on which $G$ acts by polyhedral
    automorphisms without inversions.
\end{rem}

\begin{defn}
    A \emph{$G$-simplicial complex} is a $G$-polyhedral complex, where all cells $P_i$, $i\in I_n$ are simplices.
\end{defn}

\begin{prop}%
    \label{prop:polyhedral-haefliger-model}
    Let $X$ be a cocompact $G$\=/simplicial 1-complex and let $n\geq 1$. Assume that, $G_v$ is of type $\F_{n}$ for every vertex $v \in X$ and $G_e$ is of type $\F_{n-1}$ for every edge $e$ of $X$.
    Then $\widehat{X}$ can be constructed such that $\widehat{X}^{(n)}$ is a cocompact $G$\=/polyhedral complex and
    \begin{refenum}
        \item\label{prop:polyhedral-haefliger-model:polyhedral} $\pi|_{\widehat{X}^{(n)}}$ is a polyhedral map;
        \item\label{prop:polyhedral-haefliger-model:injective} for every incidence $v \subset e \subset X$ of a vertex and an edge, the attaching map $\transGlue{v}{e}\colon EG_e^{(n-1)}\to EG_v^{(n-1)}$ is a simplicial embedding;
        \item\label{prop:polyhedral-haefliger-model:product-cells} for every edge $e\subset X$, each cell of $\widehat{X}^{(n)}$ whose interior projects onto $e^\circ$ is of the form $P\times e$, where $P$ is a cell of $EG_e$ of dimension at most $n-1$.
    \end{refenum}
\end{prop}
\begin{proof}
    Let $Y = \lquot{X}{G}$.
    Since $X$ is cocompact, $Y$ has only finitely many cells.
    For every $p$-cell $\sigma$ of $Y$ with $p \in \{0,1\}$, choose a model $EG_\sigma$ whose $(n-p)$-skeleton is a cocompact free $G_\sigma$-simplicial complex.
    Such models can be obtained by iterative simplicial approximation of the cell attaching maps.
    For a vertex $v$ of an edge $e$, one can similarly arrange that the map $\transGlue{v}{e}$ is simplicial on the relevant skeleta.
    Moreover, one can arrange that $\transGlue{v}{e}$ is injective by passing to mapping cylinders.
    Running the Haefliger construction with these maps as input data gives a complex $\widehat{X}$ satisfying \localref{prop:polyhedral-haefliger-model:injective}; see~\cite{drutu+kapovich}*{Section 5.8.2} for details.
    To obtain \localref{prop:polyhedral-haefliger-model:polyhedral}, \localref{prop:polyhedral-haefliger-model:product-cells}, and cocompactness of $\widehat{X}^{(n)}$, choose the polyhedral structure on $G\times_{G_e}EG_e\times e$ to be induced by the product cells $P\times e$, where $P$ is a cell of $EG_e$.
\end{proof}

\begin{rem}%
    \label{rem:polyhedral-model-1-skeleton-refined}
    One can strengthen the construction from \cref{rem:polyhderal-model-1-skeleton} to also satisfy the conditions \localref{prop:polyhedral-haefliger-model:injective} and \localref{prop:polyhedral-haefliger-model:product-cells} of \cref{prop:polyhedral-haefliger-model} for $n=2$ using the same techniques.
    This is all we need for the proof of \cref{thm:Dehn-Brown-trees}.

    For the proof of \cref{thm:Dehn-Brown-complexes} we will have to work with a 2-dimensional complex $X$. In this case we replace the Haefliger construction by a more direct construction of a simply connected singular combinatorial complex $\widehat{X}$ together with a singular combinatorial map $\pi\colon \widehat{X}\to X$, which we will describe now.

    Start with a complex $Y$ obtained by applying \cref{prop:polyhedral-haefliger-model} to the restriction of the $G$-action to $X^{(1)}$, that is, copies of $EG_v$ for every vertex $v \in X$ and $EG_e \times e$ for every edge $e$, glued
    together using the maps $\transGlue{v}{e}$.

    We now attach a $G_{\sigma}$-orbit of hexagonal $2$-cells $gP_{\sigma}$ for every $2$-cell $\sigma$ of ${X}$. To do so identify $EG_{\sigma}^{(0)}$ with $G_{\sigma}$ and for every edge $e\subset \sigma$ fix a $G_{\sigma}$-equivariant inclusion $\psi_{e,\sigma}\colon EG_{\sigma}^{(0)}\to EG_e^{(0)}$ induced by $\transConj{e}{\sigma}$. For the edges $e,e'\subset \sigma$ adjacent to a vertex $v\subset\sigma$, the compositions $\psi_{v,e}\circ\psi_{e,\sigma}$ and $\psi_{v,e'}\circ\psi_{e',\sigma}$ differ by an element $g_{v,\sigma}\in G_v$ only depending on the orbit of $\sigma$.
    Note that the elements $g_{v,\sigma}$ are products of monodromy elements $g_{v,e,\sigma}$. Since $\lquot{G}{X}$ is finite, we may assume that the finite generating set $S_v$ contains all such elements $g_{v,\sigma}$. For $g\in G_{\sigma}$ we attach the hexagonal $2$-cell $gP_\sigma$ along the loop described by the edges corresponding to the $g_{v,\sigma}$ inside $EG_v$ for the vertices $v\subset \sigma$ and the edges $\psi_{e,\sigma}(g)\times e$ inside $EG_e \times e$ for the edges $e\subset \sigma$ as depicted in \cref{fig:hexagonal-cell}.

    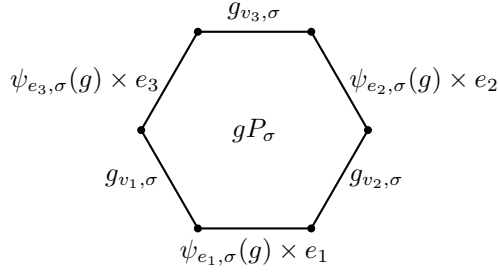
\begin{figure}
        \centering
        \begin{tikzpicture}[
        thick,
        every circle/.style={radius=1.5pt}
    ]
    \foreach \i in {0,...,5}
    \coordinate (A\i) at ({60*\i}:1.5);

    \draw (A0) \foreach \i in {1,...,5} { -- (A\i) } -- cycle;

    \foreach \i in {0,...,5}
    \fill (A\i) circle; 

    \node at (barycentric cs:A3=1,A0=1) {$gP_\sigma$};

    \node[below] at (barycentric cs:A4=1,A5=1) {$\transGlue{e_1}{\sigma}(g) \times e_1$};
    \node[left] at (barycentric cs:A3=1,A2=1) {$\transGlue{e_3}{\sigma}(g) \times e_3$};
    \node[right] at (barycentric cs:A0=1,A1=1) {$\transGlue{e_2}{\sigma}(g) \times e_2$};

    \node[left] at (barycentric cs:A4=1,A3=1) {$g_{v_1,\sigma}$};
    \node[right] at (barycentric cs:A5=1,A0=1) {$g_{v_2,\sigma}$};
    \node[above] at (barycentric cs:A2=1,A1=1) {$g_{v_3,\sigma}$};
\end{tikzpicture}
        \caption{%
            \label{fig:hexagonal-cell}
            The hexagonal cell $gP_\sigma$ with its boundary labels.
        }
    \end{figure}

    By construction $\widehat{X}$ is a singular combinatorial complex and the singular combinatorial projection $\pi\colon \widehat{X}\to X$ is the obvious extension of the projection of $Y$ to $X^{(1)}$ by the projections $gP_{\sigma}\to \sigma$. Finally, for simple connectivity of $\widehat{X}$ note that this is a direct consequence of \cref{thm:filling tree diagrams,prop:reduction to trees} which form part of the proof of \cref{thm:Dehn-Brown-complexes} and do not rely on it.
\end{rem}

\begin{defn}
    A $G$-polyhedral complex $X$ is called \emph{locally finite in degree $n$} if, for every
    $n$-cell $\sigma$ of $X$ and every $(n-1)$-face $\tau$ of $\sigma$, the index $[G_\tau : G_\sigma]$
    is finite.
\end{defn}

\begin{example}%
    \label{ex:bass-serre-locally-finite}
    Let $\calG = (\Gamma, G_e, G_v)$ be a graph of groups with fundamental group $G = \pi_1(\calG)$
    and corresponding Bass--Serre tree $T$.
    Then $T$ is a $G$-polyhedral complex and is cocompact if and only if $\Gamma$
    is a finite graph, as $\lquot{T}{G} = \Gamma$.

    For each directed edge $e = (u,v)$ of $T$, consider the edge $p(e) = (p(u),p(v))$.
    Then $[G_v : G_e] = [G_{p(v)} : \iota_{p(e)}(G_{p(e)})]$ and
    $T$ is locally finite in degree $1$ if and only if $T$ is locally finite as a graph.
    Since $T$ is $1$-dimensional, there are no higher-dimensional cells, so $T$ is locally finite in
    degree $n$ for every $n \geq 2$.
\end{example}

\section{A Brown Theorem for one-dimensional Dehn functions}\label{sec:dimension-one}

Throughout this section, let $X$ be a $G$-simplicial complex that is $1$-connected, cocompact, and locally finite in degree $2$.
We further assume that all vertex stabilizers are finitely presented and all edge stabilizers are finitely generated.

Because $X$ is cocompact, $\lquot{X}{G}$ has finite $1$-skeleton.
Let $\widehat{X}$ and $\pi \colon \widehat{X}\to X$ be as described in \cref{rem:polyhedral-model-1-skeleton-refined}.
We write $\edgeDist_X$ for the maximum of the distortion functions $\dist^{G}_{G_e}$
over the edges $e$ in $X$. Recall that, as in \cref{rem:description-haefliger-construction}, we fix a representative of the conjugacy class
of every cell stabilizer, which we denote by $G_\sigma$. Since there are only finitely many orbits of vertices and edges, $\edgeDist_X$ is well-defined, and similar considerations apply to the other maxima below.

\subsection{Lifting fillings of tree van~Kampen diagrams}

We start with the special case in which the projection of a combinatorial loop in $\widehat{X}$ to $X$ admits a van~Kampen diagram in $X$ with no $2$-cells or, equivalently, a van~Kampen diagram whose underlying complex is a tree.
Note that if $X$ is $1$-dimensional, i.e., a Bass--Serre tree, then this holds for all van~Kampen diagrams.
Thus, the results in this section produce an upper bound on the Dehn function of a graph of groups.

\begin{lem}[Edge collapse]%
    \label{lem:tree collapse single cell}
    Let $\gamma \colon [0,1] \to \widehat{X}$ be a combinatorial path of length $n$.
    Let $e = \{u,v\}$ be an edge of $X$.
    Assume $\gamma$ decomposes as $\gamma = \nu_1 \cdot \mu \cdot \nu_2$, where $\pi(\gamma(0)) = \pi(\gamma(1)) = u$, $\pi \circ \mu = \Const_v$, and $\ell(\nu_1) = \ell(\nu_2) = 1$.
    Then $\gamma$ is homotopic to a path $\gamma' \colon [0,1] \to \widehat{X}$ via a homotopy $h$ fixing endpoints such that
    \begin{enumerate}
        \item $\ell(\gamma') = d_{G_e}(\gamma(0),\gamma(1))$;
        \item $h$ has at most $\Dehn_{G_v}\bigl(\ell(\gamma) + \ell(\gamma')\bigr) + \ell(\gamma')$ nonsingular $2$-cells.
    \end{enumerate}
\end{lem}
\begin{proof}
    By assumption, $\nu_1$ and $\nu_2$ define horizontal edges in $EG_e\times[0,1]$. Let $\gamma'_0\colon[0,1]\to EG_e\times \left\{0\right\}$ be a combinatorial path of minimal length from $\nu_1(0)=\gamma(0)$ to $\nu_2(1)=\gamma(1)$.
    Then $\ell(\gamma'_0) = d_{G_e}(\gamma(0),\gamma(1))$.
    Denote by $\gamma'_1$ the parallel copy of $\gamma'_0$ inside $EG_e\times \left\{1 \right\}$.
    Then $\gamma'_1(0)=\mu(0)$ and $\gamma'_1(1)=\mu(1)$.
    Thus, there is a homotopy from $\mu$ to $\gamma'_1$ with at most $\Dehn_{G_v}\bigl(\ell(\gamma) + \ell(\gamma'_0)\bigr)$ $2$-cells.
    Moreover, parallel translation of $\gamma'_1$ to $\gamma'_0$ inside $EG_e\times[0,1]$ defines a homotopy from $\nu_1\cdot \gamma'_1\cdot \nu_2$ to $\gamma'_0$ with at most $\ell(\gamma'_0)$ $2$-cells (see \cref{fig:collapse-edge}).
    Composing these two homotopies and choosing $\gamma'\defq\gamma'_0$ completes the proof.
    \begin{figure}
        \centering
        \begin{tikzpicture}[baseline={(Center)}]
    \baseImage[%
        circleradius=8mm,
        intervalthickness=10mm,
        pathoffset=0.35*\intervalthickness,
    ]
    \ComputeCoordinates
    \begin{scope}[
            every path/.style={thick interval},
            every node/.style={midway,sloped,font=\small},
        ]
        \draw (A) -- (B);
    \end{scope}
    \begin{scope}[
            fill=white,draw=black,
            line width=\outlinethickness,
            every circle/.style={radius=\circleradius}
        ]
        \filldraw (A) circle;
        \filldraw (B) circle;
    \end{scope}

    \path[name intersections={of=ACircle and RayBAInner,name=ABInnerInt}];
    \path[name intersections={of=BCircle and RayABInner,name=BAInnerInt}];
    \path[name intersections={of=ACircle and RayBAOuter,name=ABOuterInt}];
    \path[name intersections={of=BCircle and RayABOuter,name=BAOuterInt}];

    \coordinate (ABInner) at (ABInnerInt-1);
    \coordinate (BAInner) at (BAInnerInt-2);
    \coordinate (ABOuter) at (ABOuterInt-2);
    \coordinate (BAOuter) at (BAOuterInt-1);

    \begin{scope}[red]
        \begin{scope}[loop path]
            \draw[dashed,gray] (ABInner) -- ($(ABInner)!-2mm!(BAInner)$) to[out=180,in=0] ($(A)+(90:0.7*\circleradius)$);
            \draw[dashed,gray] (ABOuter) -- ($(ABOuter)!-2mm!(BAOuter)$) to[out=180,in=0] ($(A)+(-90:0.7*\circleradius)$);
            \draw (ABInner) -- (BAInner);
            \draw (ABOuter) -- (BAOuter);
            \draw[dashed] (BAInner)
            to[out=0,in=180] ($(B)+(90:0.7*\circleradius)$)
            to[out=0,in=90] ($(B)+(0:0.5*\circleradius)$)
            node[anchor=west,xshift=-1pt] {$\mu$}
            to[out=-90,in=0] ($(B)+(-90:0.7*\circleradius)$)
            to[out=180,in=0] (BAOuter);

            \draw[connection path] (ABInner) -- (ABOuter);
            \draw[connection path] (BAInner) -- (BAOuter);
        \end{scope}
        \fill (ABInner) circle;
        \fill (BAInner) circle;
        \fill (ABOuter) circle;
        \fill (BAOuter) circle;
    \end{scope}

    \begin{scope}[font=\small]
        \node[anchor=south west] at ($(B)+(45:\circleradius)$)[inner sep=1pt] {$EG_{v}$};
        \node[anchor=south east] at ($(A)+(135:\circleradius)$)[inner sep=1pt] {$EG_{u}$};
        \path (A) --node[above=0.5*\intervalthickness] {$EG_{e} \times [0,1]$} (B);
    \end{scope}

    \begin{scope}[red]
        \node[below] at (barycentric cs:ABInner=1,BAInner=0.25) {$\nu_1$};
        \node[above] at (barycentric cs:ABOuter=0.75,BAOuter=1) {$\nu_2$};
        \node[left] at (barycentric cs:ABInner=1,ABOuter=1) {$\gamma_0$};
        \node[left] at (barycentric cs:BAInner=1,BAOuter=1) {$\gamma_1$};
        \node at (B) {$\Dehn_{G_v}$};
    \end{scope}
\end{tikzpicture}
        \caption{%
            \label{fig:collapse-edge}
            Lift of the collapse of an edge $\{u,v\}$ to a homotopy in $\widehat{X}$. The loop $\mu \cdot \gamma_1$
            is filled using the Dehn function $\Dehn_{G_v}$ of the vertex stabilizer $G_v$.
        }
    \end{figure}
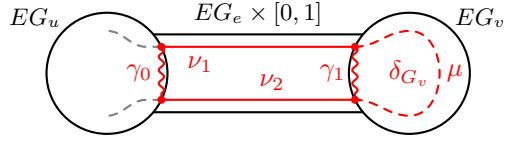
\end{proof}

\begin{rem}
    Let $g_0$ and $g_1$ be the elements in $G_e$ represented by $\gamma'(0) = \gamma(0)$ and $\gamma'(1) = \gamma(1)$
    respectively. Then
    \begin{align*}
        d_{\widehat{X}_e}(\gamma'(0),\gamma'(1)) = d_{G_e}(g_0, g_1) \leq \dist^{G}_{G_e}\bigl(d_{G}(g_0,g_1)\bigr).
    \end{align*}
    Moreover, $G$ and $\widehat{X}^{(1)}$ are quasi-isometric, so there
    exists a constant $D = D(X) \geq 1$ such that $d_{G}(g_0,g_1) \leq D \cdot d_{\widehat{X}}(\gamma(0),\gamma(1))$. Consequently,
    \begin{equation}%
        \label{eq:estimate-gamma-prime}
        \ell(\gamma')
        \leq \dist^{G}_{G_e}\bigl(D \cdot d_{\widehat{X}}(\gamma(0),\gamma(1))\bigr)
        \leq \dist^{G}_{G_e}\bigl(D \cdot \ell(\gamma)\bigr).
    \end{equation}
\end{rem}

\begin{thm}%
    \label{thm:filling tree diagrams}
    Let $\gamma \colon [0,1] \to \widehat{X}$ be a combinatorial path of length $n$.
    Assume that $\pi \circ \gamma$ admits a van~Kampen diagram $(T, f, \iota)$
    where $T$ is a tree. Then there exists a filling of $\gamma$ consisting of
    at most $n \cdot \bigl(\edgeDist_X(Dn) + \max_{v \in X^{(0)}} \bigl(\Dehn_{G_v}(\overline{\edgeDist_X}(Dn))\bigr)\bigr)$
    $2$-cells.
\end{thm}
\begin{proof}
    For ease of notation, we identify the vertices and edges of $T$ with their images under $f$ in $X$.

    Fix a sequence of edges along which $T$ is collapsed to a single vertex $v_0$.
    Let $v \in T^{(0)}$ be a vertex, let $e_1,\ldots,e_r$ be the edges adjacent to $v$ in the order in which they are collapsed, and denote by $T_i$ the subtree extending outwards from $v$ along the edge $e_i$.
    This situation is depicted in \cref{fig:subtrees around vertex}.

    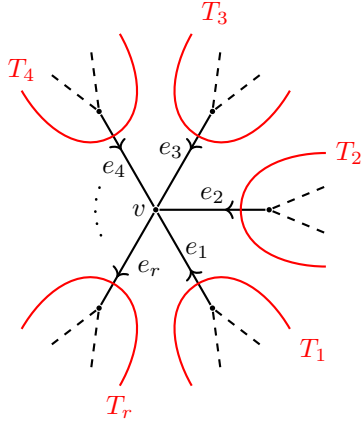
\begin{figure}
        \centering
        \tikzset{
    normal edge/.style={
            edge from parent/.style={draw,thick}
        },
    mid child arrow out/.style={
            edge from parent/.style={
                    draw,
                    postaction={
                            decorate,
                            decoration={
                                    markings,
                                    mark=at position .7 with {\arrow{>}}
                                }
                        }
                }
        },
    mid child arrow in/.style={
            edge from parent/.style={
                    draw,
                    postaction={
                            decorate,
                            decoration={
                                    markings,
                                    mark=at position .7 with {\arrow{<}}
                                }
                        }
                }
        }
}
\begin{tikzpicture}[
        bullet/.style={
                circle,
                fill=black,
                inner sep=0pt,
                minimum size=2pt
            },
        normal edge,
        level distance=15mm,
        level 1/.style={
                sibling angle=60
            },
        level 2/.style={
                dashed,
                sibling angle=45,
                level distance=8mm,
            },
        grow cyclic,
        thick
    ]
    \node[bullet] (A) {}
    child[mid child arrow out] { node[bullet] (BR) {}
            child[normal edge] {}
            child[normal edge] {}
        }
    child[mid child arrow in] { node[bullet] (B1) {}
            child[normal edge] {}
            child[normal edge] {}
        }
    child[mid child arrow in] { node[bullet] (B2) {}
            child[normal edge] {}
            child[normal edge] {}
        }
    child[mid child arrow in] { node[bullet] (B3) {}
            child[normal edge] {}
            child[normal edge] {}
        }
    child[mid child arrow in] { node[bullet] (B4) {}
            child[normal edge] {}
            child[normal edge] {}
        };
    \draw[line width=1pt, line cap=round, dash pattern=on 0pt off 4.5\pgflinewidth] (205:8mm) arc[start angle=205, end angle=155, radius=8mm];
    \node[anchor=east] at (A) {$v$};
    \path
    (A) --node[midway,anchor=north east,inner sep=0pt] {$e_4$} (B4)
    (A) --node[midway,anchor=east,inner sep=0pt,yshift=4pt] {$e_3$} (B3)
    (A) --node[midway,anchor=south,yshift=-1pt] {$e_2$} (B2)
    (A) --node[midway,anchor=south west,inner sep=0pt] {$e_1$} (B1)
    (A) --node[midway,anchor=north west,xshift=0pt,yshift=3pt] {$e_r$} (BR);
    \begin{scope}[red]
        \begin{scope}[scale=1.5]
            \foreach \B/\ang in {BR/-120,B1/-60,B2/0,B3/60,B4/120} {
                    \begin{scope}[shift={(\B)}, rotate=\ang]
                        \draw
                        (0.5,-0.5) to[out=180,in=-90] (-0.25,0) to[out=90,in=180] (0.5,0.5) coordinate (T\B);
                    \end{scope}
                }
        \end{scope}
        \node[anchor=north] at (TBR) {$T_r$};
        \node[anchor=north west] at (TB1) {$T_1$};
        \node[anchor=west] at (TB2) {$T_2$};
        \node[anchor=south west] at (TB3) {$T_3$};
        \node[anchor=south] at (TB4) {$T_4$};
    \end{scope}
\end{tikzpicture}
        \caption{%
            \label{fig:subtrees around vertex}
            The subtrees $T_i$ around a vertex $v$. Arrows indicate the direction of the collapse of the adjacent edges.
        }
    \end{figure}
    For $1 \leq i \leq r - 1$ let $\gamma_{e_i}$ be the subpath of $\gamma$ corresponding to $T_i$ and $\gamma'_{e_i}$ be
    the path from $\gamma_{e_i}(0)$ to $\gamma_{e_i}(1)$ obtained when applying \cref{lem:tree collapse single cell} to collapse the edge $e_i$ connecting $T_i$ to $v$ in our iterative procedure.
    Then
    \begin{align*}
        \ell(\gamma'_{e_i})
        \leq \dist^{G}_{G_{e_i}}\bigl(D \cdot \ell(\gamma_{e_i})\bigr),
    \end{align*}
    by \cref{eq:estimate-gamma-prime}.
    The same holds for the path $\gamma'_{e_r}$ when collapsing $e_r$.
    Therefore, the homotopy we obtain from \cref{lem:tree collapse single cell} when we collapse the edge $e_r$ from $v$ has at most
    \begin{align*}
        \Dehn_{G_v}\Bigl(\ell(\gamma) + \sum_{i=1}^{r} \dist^{G}_{G_{e_i}}\bigl(D \cdot \ell(\gamma_{e_i})\bigr)\Bigr)
         &\leq \Dehn_{G_v}\Bigl(\ell(\gamma) + \sum_{i=1}^{r} \edgeDist_X\bigl(D \cdot \ell(\gamma_{e_i})\bigr)\Bigr) \\
         &\leq \Dehn_{G_v}\Bigl(\ell(\gamma) + \overline{\edgeDist_X}\Bigl(\sum_{i=1}^{r} D \cdot \ell(\gamma_{e_i})\Bigr)\Bigr) \\
         &\leq \Dehn_{G_v}\Bigl(n + \overline{\edgeDist_X}(D n)\Bigr)
    \end{align*}
    $2$-cells coming from the filling inside $EG_v$ and at most
    \begin{align*}
        \ell(\gamma'_{e_r})
        \leq \dist^{G}_{G_{e_r}}\bigl(D \cdot \ell(\gamma_{e_r})\bigr)
        \leq \edgeDist_X(D \cdot\ell(\gamma_{e_r}))
        \leq \edgeDist_X(D \cdot \ell(\gamma))
    \end{align*}
    $2$-cells coming from the collar inside $EG_{e_r}$.
    After iteratively collapsing $T$ to a single vertex $v_0$, we have obtained a homotopy $h$ of $\gamma$ to
    a path $\gamma'$ such that $\pi \circ \gamma' = \Const_{v_0}$.
    The same analysis as above shows that $\ell(\gamma') \leq n + \overline{\edgeDist_X}(n)$.
    Filling this loop inside $EG_{v_0}$ completes the homotopy $h$ to a filling of $\gamma$.
    Because $T$ has at most $\frac{n}{2}$ edges, the resulting filling has area at most
    \begin{align*}
         &\mathbin{\phantom{=}}\frac{n}{2} \cdot \Bigl(\max_{v \in X^{(0)}} \Dehn_{G_v}\bigl(n + \overline{\edgeDist_X}(Dn)\bigr) + \overline{\edgeDist_X}(Dn)\Bigr) + \Dehn_{G_{v_0}}\bigl(n + \overline{\edgeDist_X}(Dn)\bigr) \\
         &\leq n \cdot \Bigl(\max_{v \in X^{(0)}} \Dehn_{G_v}\bigl(n + \overline{\edgeDist_X}(Dn)\bigr) + \overline{\edgeDist_X}(Dn)\Bigr).\qedhere
    \end{align*}
\end{proof}

\subsection{Intuition for the case of trees}

Before proceeding, we give an intuitive explanation of the proof when the projection $\pi\circ \gamma$ to $X$ of a combinatorial loop $\gamma\colon [0,1]\to \widehat{X}$ admits a van~Kampen diagram with no $2$-cells. Due to our choice of $\widehat{X}$, we can partition the edges defining $\gamma$ into horizontal edges, which project onto edges in $X$ under $\pi\circ \gamma$, and vertical edges, which project onto vertices in $X$ under $\pi\circ \gamma$.

If we choose a van~Kampen diagram for $\pi\circ\gamma$ with no $2$-cells, then its underlying complex is a tree and $\pi\circ\gamma$ corresponds to a parametrization of its boundary. This parametrization traverses every edge of the van~Kampen diagram precisely twice in opposite directions, inducing a pairing between the edges that define $\pi\circ \gamma$ and thus between the horizontal edges of $\gamma$.

For every pair of edges $e_+$, $e_-$ of $\gamma$, there is an edge $e$ of $X$ and elements $g_+,g_-\in G_e$ such that $\gamma|_{e_+}$ and $\gamma|_{e_-}$ parametrize edges of the form $\left\{g_+\right\}\times [0,1]$ and $\left\{g_-\right\}\times [0,1]$ in $EG_e\times [0,1]$. There is a path $\mu_e\colon [0,1]\to EG_e$ between $g_+$ and $g_-$ of length $\ell(\mu_e)\leq \dist_{G_e}^G(g_+,g_-)$.

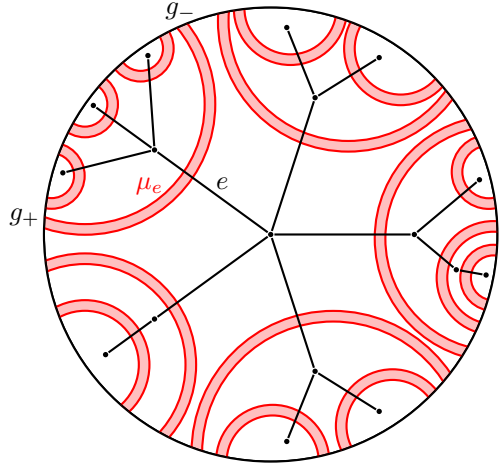
\begin{figure}
    \centering
    \def\radius{3cm}%
\newcommand{\Connection}[2]{
    (#1:\radius) circle[radius=2*\radius*sin(abs(#2)/2)-2pt]
    (#1:\radius) circle[radius=2*\radius*sin(abs(#2)/2)+2pt]
}
\begin{tikzpicture}[
        bullet/.style={
                circle,
                fill=black,
                inner sep=0pt,
                minimum size=2pt
            },
        edge from parent/.style={draw,thick},
        level 1/.style={
                sibling angle=360/5
            },
        grow cyclic,
        thick
    ]

    \node at (176:\radius+0.75em) {$g_{+}$};
    \node at (112:\radius+0.5em) {$g_{-}$};

    \draw (0,0) circle[radius=\radius];
    \clip (0,0) circle[radius=\radius];

    \begin{scope}[draw=red,even odd rule,fill=red,fill opacity=0.25]
        \filldraw
        \Connection{75+90}{7}
        \Connection{75+70}{7}
        \Connection{75+50}{7}
        \Connection{75+70}{32};
        \filldraw
        \Connection{85}{13}
        \Connection{55}{13}
        \Connection{70}{32};
        \filldraw
        \Connection{-10-1}{7}
        \Connection{-10-1}{13}
        \Connection{17-1}{8}
        \Connection{0-1}{30};
        \filldraw
        \Connection{-85.5-4}{13}
        \Connection{-53.5-4}{13}
        \Connection{-70-4}{35};
        \filldraw
        \Connection{-145}{15}
        \Connection{-145}{27};
    \end{scope}

    \draw (0,0) circle[radius=\radius];

    \node[bullet] (A) {}
    child[level distance=19mm] { node[bullet] {}
            child[level distance=8mm] {node[bullet] {}}
        }
    child[level distance=19mm] { node[bullet] {}
            child[sibling angle=80,level distance=10mm] {node[bullet] {}}
            child[sibling angle=80,level distance=10mm] {node[bullet] {}}
        }
    child[level distance=19mm] { node[bullet] {}
            child[sibling angle=80,level distance=7.25mm] {
                    node[bullet] {}
                    child[level distance=4mm,edge from parent/.style={draw=none}] {}
                    child[sibling angle=60,level distance=4mm] {node[bullet] {}}
                }
            child[sibling angle=80,level distance=11.25mm] {node[bullet] {}}
        }
    child[level distance=19mm] { node[bullet] {}
            child[sibling angle=80,level distance=10mm] {node[bullet] {}}
            child[sibling angle=80,level distance=10mm] {node[bullet] {}}
        }
    child[level distance=19mm] { node[bullet] (B) {}
            child[sibling angle=50,level distance=12.5mm] {node[bullet] {}}
            child[sibling angle=50,level distance=10mm] {node[bullet] {}}
            child[sibling angle=50,level distance=12.5mm] {node[bullet] {}}
        };
    \path (A) --node[midway,anchor=south west,inner sep=1pt] {$e$} (B);
    \node[red,anchor=south east,inner sep=0pt] at (160:0.5*\radius) {$\mu_e$};
\end{tikzpicture}
    \caption{%
        \label{fig:disk-partition}
        A tree van~Kampen diagram and the corresponding disk partition obtained by the connecting paths of the pairings $\{g_+,g_-\}$.
    }
\end{figure}

To build a van~Kampen diagram for $\gamma$, we label the boundary of a $2$-disc by the edges defining $\gamma$ and subdivide it along corridors between paired edges $e_+,e_-$ labeled by the paths $\mu_e$ (see \cref{fig:disk-partition}). The boundary of every bounded region in this disc describes a word of length $\leq \edgeDist (\ell(\gamma))$ in some vertex group $G_v$. This implies that there is a van~Kampen diagram for $\gamma$ satisfying the upper area bound in \cref{thm:filling tree diagrams}.

\subsection{Reduction to trees}

If $X$ is not $1$-dimensional, we have to consider loops in $X$ that do not admit a van~Kampen diagram with no $2$-cells. We will deal with this case by reducing it to the case of a van~Kampen diagram that is a tree.
For this, we use \cref{lem:planaer-complex-collapsible}, which states that any finite contractible planar simplicial $2$-complex can be collapsed onto a tree by a sequence of elementary collapses of $2$-simplices along codimension-one faces.

\begin{lem}[Face collapse]%
    \label{lem:lifting-elementary-2-simplex-homotopies}
    Let $\gamma\colon [0,1]\to \widehat{X}$ be a combinatorial path. Let $e=\left\{v_1,v_2\right\}$ be an edge of $X$ contained in a $2$-simplex $\sigma=\left\{v_1,v_2,v_3\right\}$.
    Assume that $\pi(\gamma(0))=v_1$, $\pi(\gamma(1))=v_2$, and $\ell(\gamma)=1$.

    Then $\gamma$ is homotopic to a path $\gamma'=\mu_1\cdot \lambda_1 \cdot \nu_1\cdot\lambda_3 \cdot \nu_2\cdot \lambda_2 \cdot \mu_2 \colon [0,1]\to  \widehat{X}$ via a homotopy $h$ fixing endpoints such that
    \begin{enumerate}
        \item $\pi\circ \mu_i=\Const_{v_i}$ for $i \in \{1,2\}$, and $\pi\circ\lambda_i=\Const_{v_i}$ for $i \in \{1,2,3\}$;
        \item $\pi\circ\nu_1$ parametrizes $\{v_1,v_3\}$ and $\pi\circ\nu_2$ parametrizes $\{v_3,v_2\}$;
        \item $\nu_1$ and $\nu_2$ have length $1$;
        \item $\ell(\lambda_i)\leq 1$ for $i \in \{1,2,3\}$;
        \item $\ell(\mu_1),~\ell(\mu_2)\leq [G_e : G_{\sigma}]-1$;
        \item $h$ has at most $[G_e : G_{\sigma}]$ many $2$-cells.
    \end{enumerate}
\end{lem}

\begin{proof}
    Since $[G_e : G_{\sigma}]<\infty$, there is a path $\mu_1\colon [0,1]\to EG_e\times\{0\}$ of length $\ell(\mu_1)\leq [G_e : G_{\sigma}]-1$ such that $\mu_1(0)=\gamma(0)$ and $\mu_1(1)$ lies in the image of $EG_\sigma^{(0)}$ in $EG_e^{(0)}$.
    We choose $\mu_2$ to be the path parallel to $\mu_1$ in $EG_e\times\{1\}$ parametrized in the opposite direction.
    By \cref{rem:polyhedral-model-1-skeleton-refined}, there is, for some $g$, a $2$-cell $gP_\sigma$ over $\sigma$ whose boundary contains the edge from $\mu_1(1)$ to $\mu_2(0)$, edges $\nu_1$ and $\nu_2$ projecting onto $\{v_1,v_3\}$ and $\{v_3,v_2\}$, respectively, and possibly edges $\lambda_i$ in the vertex fibers, where $\lambda_i$ is constant if this edge is absent.
    By construction, parallel translation of $\mu_1$ inside $EG_e\times[0,1]$, together with $gP_\sigma$, defines the required homotopy from $\gamma$ to $\gamma'$ with $\ell(\mu_1)+1\leq [G_e : G_\sigma]$ $2$-cells.
    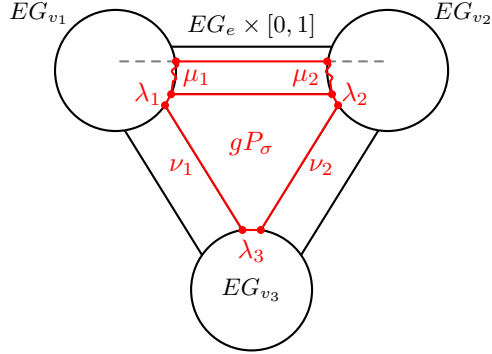
\begin{figure}
        \centering
        \begin{tikzpicture}[baseline={(Center)}]
    \baseImage[%
        sidelength=3,
        circleradius=8mm,
        intervalthickness=6mm,
        innerpathoffset=0.5*\intervalthickness+0.5*\outlinethickness,
        outerpathoffset=0.2*\intervalthickness,
        Ax={-0.6*\sidelength},
        Ay={sqrt(3)/6*\sidelength},
        Bx={0.6*\sidelength},
        By={sqrt(3)/6*\sidelength},
        Cx={0},
        Cy={-sqrt(3)/3*\sidelength-0.1*\sidelength},
    ]
    \ComputeCoordinates
    \begin{scope}[
            every path/.style={thick interval},
            every node/.style={midway,sloped,font=\small},
        ]
        \draw (A) -- (B) -- (C) -- cycle;
    \end{scope}
    \begin{scope}[
            fill=white,draw=black,
            line width=\outlinethickness,
            every circle/.style={radius=\circleradius}
        ]
        \filldraw (A) circle;
        \filldraw (B) circle;
        \filldraw (C) circle;
    \end{scope}

    \path[name intersections={of=ACircle and RayBAInner,name=ABInnerInt}];
    \path[name intersections={of=ACircle and RayCAInner,name=ACInnerInt}];
    \path[name intersections={of=BCircle and RayABInner,name=BAInnerInt}];
    \path[name intersections={of=BCircle and RayCBInner,name=BCInnerInt}];
    \path[name intersections={of=CCircle and RayACInner,name=CAInnerInt}];
    \path[name intersections={of=CCircle and RayBCInner,name=CBInnerInt}];

    \coordinate (ABInner) at (ABInnerInt-2);
    \coordinate (ACInner) at (ACInnerInt-2);
    \coordinate (BAInner) at (BAInnerInt-1);
    \coordinate (BCInner) at (BCInnerInt-2);
    \coordinate (CAInner) at (CAInnerInt-1);
    \coordinate (CBInner) at (CBInnerInt-1);

    \begin{scope}[red]
        \begin{scope}[loop path]
            \draw[dashed,gray] (ABOuter) -- ($(ABOuter)!-8mm!(BAOuter)$);
            \draw (ABOuter) -- (BAOuter);
            \draw[dashed,gray] (BAOuter) -- ($(BAOuter)!-8mm!(ABOuter)$);
            \draw[fill=white] (CAInner) -- (CBInner) -- (BCInner) -- (BAInner) -- (ABInner) -- (ACInner) -- cycle;

            \draw[connection path] (ABOuter) -- (ABInner);
            \draw[connection path] (BAOuter) -- (BAInner);
        \end{scope}
        \fill (ABOuter) circle;
        \fill (BAOuter) circle;

        \fill (ABInner) circle;
        \fill (ACInner) circle;
        \fill (BAInner) circle;
        \fill (BCInner) circle;
        \fill (CAInner) circle;
        \fill (CBInner) circle;
    \end{scope}

    \begin{scope}[font=\small]
        \node[anchor=south west] at ($(B)+(45:\circleradius)$)[inner sep=1pt] {$EG_{v_2}$};
        \node[anchor=south east] at ($(A)+(135:\circleradius)$)[inner sep=1pt] {$EG_{v_1}$};
        \node[anchor=center] at (C) {$EG_{v_3}$};
        \path (A) --node[above=0.5*\intervalthickness] {$EG_{e} \times [0,1]$} (B);
    \end{scope}

    \begin{scope}[red]
        \node[right] at (barycentric cs:ABInner=1,ABOuter=1) {$\mu_1$};
        \node[left] at (barycentric cs:BAInner=1,BAOuter=1) {$\mu_2$};
        \node[left] at (barycentric cs:ACInner=1,CAInner=1) {$\nu_1$};
        \node[right] at (barycentric cs:BCInner=1,CBInner=1) {$\nu_2$};
        \node at (barycentric cs:A=1,B=1,C=1) {$gP_\sigma$};
        \node[below] at (barycentric cs:CAInner=1,CBInner=1) {$\lambda_3$};
        \node[left] at (ABInner) {$\lambda_1$};
        \node[right] at (BAInner) {$\lambda_2$};
    \end{scope}
\end{tikzpicture}
        \caption{%
            \label{fig:collapse-face}
            Lift of the collapse of a $2$-simplex $\{v_1,v_2,v_3\}$ to a homotopy in $\widehat{X}$.
        }
    \end{figure}
\end{proof}

\begin{rem}
    If $X$ is a cocompact $G$-simplicial complex that is locally finite in degree $2$, then
    \begin{align*}
        C(X) \defq \max\left\{[G_e : G_{\sigma}]\mid \text{ $\sigma$ is a $2$-cell and $e\subset \sigma$ is an edge}\right\} -1
    \end{align*}
    is finite.
\end{rem}

\begin{prop}%
    \label{prop:reduction to trees}
    Let $\gamma\colon [0,1]\to \widehat{X}$ be a combinatorial loop of length $n$.
    Let $(\Delta,f,\iota)$ be a van~Kampen diagram for $\pi\circ \gamma$.
    Let $C \defq C(X)$. Then $\gamma$ is homotopic to a loop $\gamma'\colon[0,1]\to \widehat{X}$ via a homotopy $h$ that fixes the basepoint such that
    \begin{enumerate}
        \item $\pi\circ\gamma'$ has a van~Kampen diagram that is a tree;
        \item $\ell(\gamma')\leq \ell(\gamma)+ (2C+4) \cdot \Area(f)$;
        \item $h$ has at most $\Area(f) \cdot (C+1)$ $2$-cells.
    \end{enumerate}
\end{prop}
\begin{proof}
    Since $\Delta$ is a contractible planar simplicial $2$-complex, it is collapsible. Thus, by \cref{lem:collapse-in-order}, there is a sequence of $\Area(f)$ elementary collapses
    \[
        \Delta=\Delta_0 \collapse \Delta_1 \collapse \ldots \collapse \Delta_{\Area(f)}.
    \]
    Set $\Delta'\defq\Delta_{\Area(f)}$.
    This sequence transforms $\Delta$ into a tree $\Delta'$ and leaves the vertex set of the diagram invariant. By \cref{lem:lifting-elementary-2-simplex-homotopies}, it induces a sequence of basepoint-preserving homotopies $h_1,\ldots, h_{\Area(f)}$ between loops $\gamma=\gamma_0,\gamma_1,\ldots,\gamma_{\Area(f)}$.
    Set $\gamma'\defq\gamma_{\Area(f)}$. Then
    \begin{itemize}
        \item $\pi\circ\gamma'$ parametrizes the boundary of $\Delta'$ and thus admits a van~Kampen diagram that is a tree;
        \item $\ell(\gamma_{i+1})\leq \ell(\gamma_i) + 2C+4$;
        \item $h_i$ has at most $C+1$ many $2$-cells.
    \end{itemize}
    Choosing $h$ to be the concatenation of the homotopies $h_1,\ldots,h_{\Area(f)}$ thus completes the proof.
\end{proof}

We can now prove \cref{thm:Dehn-Brown-complexes,thm:Dehn-Brown-trees}.
\begin{proof}[Proof of \cref{thm:Dehn-Brown-complexes}]
    Up to $\approx$-equivalence we may assume that all functions
    $\Dehn_X$, $\Dehn_{G_v}$, and $\edgeDist_X$ are at least linear. It follows from \cref{prop:reduction to trees,thm:filling tree diagrams} that $\widehat{X}$ is simply connected and by construction it admits a free cocompact cellular $G$-action. Hence, $\Dehn_{\widehat{X}} \approx \Dehn_G$.

    Together, \cref{prop:reduction to trees,thm:filling tree diagrams} yield
    that
    \begin{align*}
        \Dehn_G(n)
         &\leq m_n \cdot \Bigl(\max_{v \in X^{(0)}} \Dehn_{G_v}\bigl(m_n + \overline{\edgeDist_X}(Dm_n)\bigr) + \edgeDist_X(Dm_n)\Bigr)
        + \Dehn_X(n) (C + 1),
    \end{align*}
    where $m_n = n + (2C + 4) \cdot \Dehn_X(n)$.
    As $m_n \preccurlyeq \Dehn_X(n)$, it follows that
    \begin{align*}
        m_n + \overline{\edgeDist_X}(Dm_n)
        \preccurlyeq \overline{\edgeDist_X}(Dm_n)
        \preccurlyeq \overline{\edgeDist_X}(\Dehn_X(n))
    \end{align*}
    and
    \begin{align*}
        \Dehn_{G_v}\bigl(m_n + \overline{\edgeDist_X}(Dm_n)\bigr) + \edgeDist_X(Dm_n)
         &\preccurlyeq
        \Dehn_{G_v}\bigl(\overline{\edgeDist_X}(\Dehn_X(n))\bigr),
    \end{align*}
    so the claimed upper bound follows.
\end{proof}

\begin{proof}[Proof of \cref{thm:Dehn-Brown-trees}]
    This is the special case of \cref{thm:Dehn-Brown-complexes} when $X$ is a tree.
\end{proof}

\begin{rem}
    \Cref{thm:Dehn-Brown-complexes} generalizes to actions on polyhedral $2$-complexes. Indeed, the construction in \cref{rem:polyhedral-model-1-skeleton-refined} works for arbitrary $n$-gons by attaching a $2n$-cell instead of a hexagon and similarly \cref{lem:planaer-complex-collapsible,lem:lifting-elementary-2-simplex-homotopies} generalize to arbitrary $n$-gons up to enlarging the constants. We restricted ourselves to the case of simplicial complexes to simplify the exposition.
\end{rem}

\section{Higher Dehn functions}
\label{sec:high-dimension}
In this section, we explain how our proof can be adapted to obtain upper bounds on higher Dehn functions of graphs of groups. As a consequence, we prove \cref{thm:high-dimensional-Brown-trees} and \cref{cor:high-dimensional-Brown-trees}.

\subsection{Lifting fillings of spheres and the second Dehn function}\label{sec:dimension-two}
We now provide a proof of the upper bound on the second Dehn function of a group acting cocompactly on a tree. The main result of this subsection is originally due to Barnard, Brady, and Dani (see~\cite{BarnardBradyDani-2012}*{Proposition 6.1}). We include a proof for the reader's convenience.

Throughout this section, let $X$ be a cocompact $G$-simplicial tree whose vertex stabilizers are of type $\F_3$ and whose edge stabilizers are of type $\F_2$.
Let $\widehat{X}$ and $\pi \colon \widehat{X} \to X$
be the complex and map obtained from \cref{prop:polyhedral-haefliger-model} for $n = 3$, that is, $\widehat{X}^{(3)}$ is a cocompact $G$-polyhedral complex and $\pi|_{\widehat{X}^{(3)}}$ is a polyhedral map.

\begin{defn}%
    \label{defn:tree-van-kampen-diagram}
    Let $f\colon S\cong S^2\to X$ be a polyhedral map from a finite
    polyhedral $2$-sphere.
    A \emph{tree van~Kampen diagram} for $f$ is a triple $(T,F,\iota)$
    consisting of a finite tree $T$ and polyhedral maps $\iota \colon S \to T$, $F \colon T \to X$
    such that
    \begin{condenum}
        \item $F \circ \iota = f$;
        \item $\iota^{-1}(v)$ is connected for every vertex $v$ of $T$;
        \item for every edge $e$ of $T$, there is a polyhedral homeomorphism
              $\iota^{-1}(e^\circ)\cong S^1\times e^\circ$ under which $\iota$
              corresponds to projection onto $e^\circ$.
    \end{condenum}
\end{defn}

\begin{rem}%
    \label{rem:polyhedral-from-singular-fillings}
    The definition of a singular combinatorial map $f\colon S^2\to \widehat{X}$ implies that singular $1$-cells are mapped onto $0$-cells and that singular $2$-cells are mapped onto the image of their attaching maps in the $1$-skeleton, which is a polyhedral subcomplex of $\widehat{X}^{(1)}$. A suitable simplicial subdivision of the singular $2$-cells of $S^2$ thus yields a polyhedral map $f'\colon S^2\to \widehat{X}$. By construction, $f$ and $f'$ have the same image. Consequently, any filling $F'\colon D^3\to \widehat{X}$ of $f'$ induces a filling $F\colon D^3\to \widehat{X}$ of $f$ with $\Vol(F)=\Vol(F')$. Thus, we can restrict to bounding the filling volume of polyhedral maps $S^2\to \widehat{X}$ to derive an upper bound on $\Dehn_G^2$.

    The same line of argument allows us to replace singular combinatorial fillings of combinatorial loops by polyhedral fillings.
\end{rem}

\begin{prop}%
    \label{prop:tree-van-kampen-diagram-existence}
    Let $X$ be a polyhedral complex and let $f\colon S\cong S^2\to X$ be a
    polyhedral map from a finite polyhedral $2$-sphere.
    If $f(S)$ is a tree, then
    $f$ admits a tree van~Kampen diagram.
\end{prop}
\begin{proof}
    Define an equivalence relation on $S$ by declaring
    $x \sim y$ if $f(x)=f(y)$ and $x$ and $y$ lie in the same connected component of
    $f^{-1}(f(x))$.
    Let $T \defq \quot{S}{\sim}$ and denote the quotient map by $\iota \colon S \to T$.
    Since $f$ is constant on equivalence classes, there is an induced map
    $F \colon T \to X$ such that $F \circ \iota=f$.

    Since $S$ is compact and $f$ is polyhedral, $T$ is a finite graph.
    Hence, with the induced polyhedral structure on $T$, the maps $\iota$ and $F$ are polyhedral.

    Let $e$ be an edge of $T$.
    The restriction $\iota^{-1}(e^\circ) \to e^\circ$ is a PL bundle whose fibers are compact $1$-manifolds.
    Its fibers have empty boundary because $S$ has no boundary, and they are connected by construction.
    Hence, every fiber is a circle.
    Since $e^\circ$ is an interval, the bundle is trivial, so $\iota^{-1}(e^\circ) \cong S^1 \times e^\circ$, with $\iota$ corresponding to projection onto $e^\circ$.

    It remains to show that $T$ is a tree.
    It is connected because $S$ is connected.
    Suppose that $T$ contains a cycle and choose a point $t$ in the interior of an edge on this cycle.
    Then $\iota^{-1}(t)$ is an embedded circle in $S \cong S^2$, so it separates $S$ by the Jordan--Sch\"onflies theorem.
    On the other hand, $T \setminus \{t\}$ is connected.
    Since $\iota$ is monotone, the preimage of this connected subspace is connected.
    Thus, $S \setminus \iota^{-1}(t)$ is connected, a contradiction.
    It follows that $T$ has no cycles and hence is a tree.
\end{proof}

\begin{lem}%
    \label{lem:extremal-vertex-fibers}
    Let $(T,F,\iota)$ be a tree van~Kampen diagram for a polyhedral map
    $f \colon S \cong S^2 \to X$.
    Let $e=[u,v]$ be an edge of $T$ such that $v$ is a leaf of $T$.
    Then $\overline{\iota^{-1}((u,v])}$ is a polyhedral $2$-disk.
\end{lem}
\begin{proof}
    Choose $t \in (u,v)$ and set $C \defq \iota^{-1}(t)$.
    By \cref{defn:tree-van-kampen-diagram}, $C$ is an embedded circle in $S$.
    Thus, by the Jordan--Schönflies theorem, $S \setminus C$ has two components, both open $2$-discs.

    Since $v$ is a leaf, $(t,v]$ is a component of $T \setminus \{t\}$.
    The map $\iota$ is monotone, so $\iota^{-1}((t,v])$ is connected. It follows that
    $\iota^{-1}((t,v])$ is one of the two components of $S \setminus C$, and hence is
    an open $2$-disc. Its closure is therefore a closed $2$-disc with boundary $C$.

    The product structure over the open edge identifies $\iota^{-1}((u,t])$ with an
    annular collar of $C$. Hence, $\overline{\iota^{-1}((u,v])}$ is obtained from the
    closed disc $\overline{\iota^{-1}((t,v])}$ by extending this collar. Our injectivity assumption on the attaching maps implies that this is again
    a closed $2$-disc. Since $\iota$ is polyhedral, this disc is polyhedral.
\end{proof}

\begin{lem}%
    \label{lem:tree collapse single cell-2d}
    Let $f \colon D \cong D^2 \to \widehat{X}$ be a polyhedral map.
    Let $e = \{u,v\}$ be an edge in $X$ and assume that $(\pi \circ f)^{-1}(u) = \partial D$ and $\pi(f(D)) = e$.
    Then $f$ is homotopic to a polyhedral map $f' \colon D' \cong D^2 \to \widehat{X}_u$ via a homotopy $h$ fixing $\partial D$ such that
    \begin{enumerate}
        \item $\Volume(f') \leq \Dehn_{G_e}(\Volume(f|_{\partial D}))$;
        \item $h$ has at most $\Dehn_{G_v}^2\bigl(\Volume(f) + \Volume(f')\bigr) + \Volume(f')$ nonsingular $3$-cells.
    \end{enumerate}
\end{lem}
\begin{proof}
    It follows from the assumptions that the boundary loop $f|_{\partial D}$ is contained in $\widehat{X}_e \times \{u\}$.
    Let $f'_0 \colon D' \to \widehat{X}_e \times \{0\}$ be a polyhedral filling of this loop with
    $\Volume(f'_0) \leq \Dehn_{G_e}(\Volume(f|_{\partial D}))$.
    Denote by $f'_1 \colon D' \to \widehat{X}_e \times \{1\}$
    the parallel copy of $f'_0$ at the other end of the cylinder $\widehat{X}_e \times [0,1]$.

    Collapse $\widehat{X}_e \times [0,1]$ to the end attached to the vertex fiber over
    $v$. Composing $f$ with this collapse gives a polyhedral map $f_v \colon D \to \widehat{X}_v$
    whose boundary agrees with the boundary of $f'_1$. Hence, $f_v \cup f'_1$ is a
    polyhedral map from a $2$-sphere to $\widehat{X}_v$ of volume at most $\Volume(f)+\Volume(f_1')$. We fill this
    sphere inside $\widehat{X}_v$ using at most
    \[
        \Dehn_{G_v}^2\bigl(\Volume(f)+\Volume(f_1')\bigr)
    \]
    nonsingular $3$-cells.

    Finally, parallel translation of $f'_1$ back to $f'_0$ inside
    $\widehat{X}_e \times [0,1]$ gives a collar with $\Volume(f_1')$ many nonsingular $3$-cells.
    Gluing this collar to the filling inside $\widehat{X}_v$ gives a homotopy from
    $f$ to $f'\defq f'_0$ fixing $\partial D$, with the claimed bound.
\end{proof}

\begin{thm}%
    \label{thm:filling-tree-diagrams-2d}
    Let $f \colon S \cong S^2 \to \widehat{X}$ be a polyhedral map with $\Volume(f) = n$.
    Assume that $\pi \circ f$ admits a tree van~Kampen diagram $(T, F, \iota)$.
    Then there exists a filling of $f$ with
    \[
        \Vol(f)\leq \overline{\max_e\Dehn_{G_e}}(n)
        +\overline{\max_v\Dehn_{G_v}^2}
        \bigl(n+2\cdot \overline{\max_e\Dehn_{G_e}}(n)\bigr).
    \]
\end{thm}
\begin{proof}
    The proof works analogously to the proof of \cref{thm:filling tree diagrams}.
    We iteratively apply \cref{lem:tree collapse single cell-2d} to the map
    $f|_D \colon D \to \widehat{X}$, where $D = \overline{\iota^{-1}((u,v])}$ is the disk
    from \cref{lem:extremal-vertex-fibers}. Doing so gives a homotopy to a polyhedral map
    $f' \colon S' \to \widehat{X}$ such that $\pi \circ f'$ admits a tree van~Kampen diagram with
    one fewer edge than $T$.

    Using the notation from \cref{lem:tree collapse single cell-2d}, let $e_1,\ldots,e_{r_v}$ be the
    edges incident to $v$ in the order in which they are collapsed (see again \cref{fig:subtrees around vertex}).
    Denote by $f_{v,i}\colon D_{v,i}\to \widehat{X}$ the polyhedral map to which we apply \cref{lem:tree collapse single cell-2d} when collapsing the edge $e_i$, and by $f_{v,i}'\colon D_{v,i}'\to \widehat{X}$ the resulting polyhedral map.
    When collapsing the last edge $e_{r_v}$, we apply \cref{lem:tree collapse single cell-2d} to a map $f_{v,r_v}\colon D_{v,r_v}\to \widehat{X}$ defined on the $2$-disc $D_{v,r_v}=S_v\cup D_{v,1}'\cup \ldots \cup D_{v,r_v-1}'$. The resulting homotopy therefore has at most
    \begin{align*}
        \Dehn_{G_v}^2\Bigl(\Volume(f|_{S_v}) + \sum_{i=1}^{r_v} \Volume(f_{v,i}')\Bigr)
         &\leq \Dehn_{G_v}^2\Bigl(\Volume(f|_{S_v}) + \sum_{i=1}^{r_v} \max_e\Dehn_{G_e}\bigl(\Volume(f|_{\partial D_{v,i}})\bigr)\Bigr) \\
         &\leq \Dehn_{G_v}^2\Bigl(\Volume(f|_{S_v}) + \overline{\max_e\Dehn_{G_e}}\Bigl(\sum_{i=1}^{r_v} \Volume(f|_{\partial D_{v,i}})\Bigr)\Bigr)
    \end{align*}
    nonsingular $3$-cells.
    The resulting filling therefore has at most
    \begin{equation}
        \label{eq:upper-bound-3cell-vertex}
        \overline{\max_v\Dehn_{G_v}^2}\Bigl(\sum_v \Volume(f|_{S_v}) + 2\cdot \overline{\max_e\Dehn_{G_e}}\Bigl(\sum_v\sum_{i=1}^{r_v} \Volume(f|_{\partial D_{v,i}})\Bigr) \Bigr)
    \end{equation}
    nonsingular $3$-cells coming from fillings inside the $EG_v$. Observe that $\Vol(f|_{\partial D_{v,i}})$ coincides with the number of nonsingular $2$-cells in the collar above $e_i$ that extends $\partial D_{v,i}$ into the original sphere $S$. In particular, the $2$-cells of the $S_v$ and the $2$-cells of the collars defined by the $\partial D_{v,i}$ are pairwise disjoint subcomplexes of $S$. Thus,
    $\sum_v \Volume(f|_{S_v}) \leq n$ and $\sum_v \sum_{i=1}^{r_v} \Volume(f|_{\partial D_{v,i}}) \leq n$,
    and hence \cref{eq:upper-bound-3cell-vertex} is bounded above by
    $\overline{\max_v\Dehn_{G_v}^2}\bigl(n + 2\cdot \overline{\max_e\Dehn_{G_e}}(n)\bigr)$.
    Similarly, one sees that the fillings inside the $EG_e$ contribute at most
    $\overline{\max_e\Dehn_{G_e}}(n)$ nonsingular $3$-cells.
\end{proof}
We can now deduce~\cite{BarnardBradyDani-2012}*{Proposition 6.1} from \cref{thm:filling-tree-diagrams-2d} and \cref{rem:polyhedral-from-singular-fillings}:

\begin{cor}[{\cite{BarnardBradyDani-2012}*{Proposition 6.1}}]
    If $X$ is a tree, then
    \[\Dehn_G^2(n)\preccurlyeq
        \overline{\max_v\Dehn_{G_v}^2}
        \bigl(\overline{\max_e\Dehn_{G_e}}(n)\bigr).\]
\end{cor}

\subsection{Homological filling functions}
We now turn to upper bounds for the higher homological Dehn functions and the proof of \cref{thm:high-dimensional-Brown-trees}.

As before, let $X$ be a cocompact $G$-simplicial tree
and $\widehat{X}$ and $\pi \colon \widehat{X} \to X$
be the complex and map obtained from \cref{prop:polyhedral-haefliger-model} for $n = k+1$. The proof follows a similar strategy to those in \cref{sec:dimension-one,sec:dimension-two}.

\begin{defn}%
    \label{defn:tree-decomposition}
    If $\alpha =\sum_\sigma \alpha_\sigma\sigma$ is a cellular chain and $A$ is a subspace, write $\alpha|_A \defq \sum_{\sigma \subseteq A}\alpha_\sigma\sigma$
    for the restriction of $\alpha$ to $A$.

    Let $k\geq 1$, let $T\subseteq X$ be a finite subtree, and let
    $z\in Z_k(\widehat{X})$ be supported in $\pi^{-1}(T)$.
    For each cell $\sigma$ of $T$, set
    $c(\sigma)\defq z|_{\pi^{-1}(\sigma)}-z|_{\pi^{-1}(\partial\sigma)}$.
    Thus, $c(\sigma)$ consists precisely of the terms supported on cells that
    project onto $\sigma$.
    We call the family $(c(\sigma))_{\sigma\subseteq T}$ the \emph{$T$-decomposition} of $z$.
    The $T$-decomposition of $z$ is unique and satisfies $z=\sum_{\sigma\subseteq T}c(\sigma)$.

    For an edge $e$ of $T$ incident to a vertex $v$, set
    $b(e,v)\defq \partial c(e)|_{\widehat{X}_e\times\{v\}}$.
\end{defn}

\begin{lem}\AvoidPageBreak%
    \label{lem:tree-decomposition-properties}
    Let $(c(\sigma))_{\sigma\subseteq T}$ be a $T$-decomposition of a cycle $z\in Z_k(\widehat{X})$ supported on a tree $T\subseteq X$. Then
    \begin{refenum}
        \item\label{tree-decomposition-properties:edge-boundary}
              $\partial c(e)=b(e,u)+b(e,v)$ for every edge $e=\{u,v\}$;
        \item\label{tree-decomposition-properties:endpoint-cycles}
              $\partial b(e,v) = 0$ for every edge $e$ with vertex $v$;
        \item\label{tree-decomposition-properties:vertex-boundary}
              $\partial c(v)+\sum_{e\ni v}b(e,v)=0$ for every vertex $v$;
        \item\label{tree-decomposition-properties:mass}
              $\sum_{\sigma\subseteq T}\Mass{c(\sigma)}=\Mass{z}$;
        \item\label{tree-decomposition-properties:endpoint-mass}
              $\Mass{b(e,v)}\leq\Mass{c(e)}$ for every edge $e$ with vertex $v$.
    \end{refenum}
\end{lem}
\begin{proof}
    By \cref{prop:polyhedral-haefliger-model:product-cells}, every term of $\partial c(e)$ not supported over a vertex projects onto $e^\circ$.
    No boundary of another term in the $T$-decomposition projects onto $e^\circ$, so these terms form the part of $\partial z$ projecting onto $e^\circ$ and therefore vanish, proving \localref{tree-decomposition-properties:edge-boundary}.
    By \localref{tree-decomposition-properties:edge-boundary}, $\partial b(e,u) + \partial b(e,v) = 0$.
    Because $b(e,u)$ and $b(e,v)$ have disjoint supports, no cancellation can occur, implying \localref{tree-decomposition-properties:endpoint-cycles}.
    To see \localref{tree-decomposition-properties:vertex-boundary}, observe
    that
    $(\partial z)|_{\pi^{-1}(v)}
        =\partial c(v)+\sum_{e\ni v}b(e,v)=0$.

    The chains $c(\sigma)$ have pairwise disjoint cellular supports.
    Their masses therefore add without cancellation, implying \localref{tree-decomposition-properties:mass}.

    Finally, by \cref{prop:polyhedral-haefliger-model:product-cells}, every
    cell occurring in $c(e)$ is a product with $e$ and has one face over each
    endpoint.
    The simplicial embeddings from \cref{prop:polyhedral-haefliger-model:injective} make the induced endpoint maps norm non-increasing, proving \localref{tree-decomposition-properties:endpoint-mass}.
\end{proof}

\begin{rem}%
    \label{rem:leaf-collapse-decomposition}
    Let $z$ be a cellular $k$-cycle with a $T$-decomposition, and let
    $e=\{u,v\}$ be an edge of $T$ such that $v$ is a leaf. Set
    $E\defq\widehat{X}_e\times e$ and
    $E_u\defq\widehat{X}_e\times\{u\}$.
    Then $a\defq c(e)+c(v)$ is supported in
    $E\cup\widehat{X}_v$ and satisfies $a|_{E_u}=0$.
    Moreover, by \cref{lem:tree-decomposition-properties} we have $\partial a=b(e,u)+b(e,v)+\partial c(v)=b(e,u) \in C_{k-1}(E_u)$.
\end{rem}

\begin{lem}%
    \label{lem:tree collapse single cell-kd}
    Let $k\geq 2$ and let $e=\{u,v\}$ be an edge of $X$. Set
    $E\defq \widehat{X}_e\times e$, $E_u\defq \widehat{X}_e\times \{u\}$, and
    $E_v\defq \widehat{X}_e\times \{v\}$.
    Let $a\in C_k(\widehat{X})$ be a cellular $k$-chain whose support is contained
    in $E\cup\widehat{X}_v$. Assume that
    $a|_{E_u}=0$ and that
    $b\defq \partial a\in C_{k-1}(E_u)$. Then there are chains
    $a'\in C_k(E_u)$ and $H\in C_{k+1}(\widehat{X})$ such that
    \begin{enumerate}
        \item $\partial a'=b$;
        \item $\partial H=a-a'$;
        \item $\Mass{a'} \leq \Filling_{G_e}^{k}(\Mass{b})$;
        \item $\Mass{H} \leq \Filling_{G_v}^{k+1}\bigl(\Mass{a}+\Mass{a'}\bigr)+\Mass{a'}$.
    \end{enumerate}
\end{lem}
\begin{proof}
    Since $b=\partial a$, the chain $b$ is a $(k-1)$-cycle in $E_u$. Choose a
    filling $a'\in C_k(E_u)$ with $\partial a'=b$ and
    $\Mass{a'}\leq \Filling_{G_e}^{k}(\Mass{b})$.

    Put $Y\defq E\cup\widehat{X}_v$. The attaching map
    $E_v\to\widehat{X}_v$ is cellular and injective, so its
    induced map on cellular chains is norm non-increasing.

    Let $r\colon Y\longrightarrow \widehat{X}_v$ be the cellular retraction
    obtained by collapsing the interval factor of $E=\widehat{X}_e\times e$ to
    the endpoint $v$ and then using the attaching map
    $E_v\to\widehat{X}_v$. Let $a_v\defq r_\ast a'$. Then
    $\zeta\defq r_\ast a-a_v$ is a $k$-cycle in $\widehat{X}_v$, since
    $\partial\zeta=r_\ast b-r_\ast b=0$.
    Moreover, $\Mass{\zeta}\leq \Mass{a}+\Mass{a'}$.
    Hence, there is a chain
    $B\in C_{k+1}(\widehat{X}_v)$ such that $\partial B=\zeta$ and
    $\Mass{B}\leq \Filling_{G_v}^{k+1}\bigl(\Mass{a}+\Mass{a'}\bigr)$.

    Let $j\colon \widehat{X}_v\hookrightarrow Y$ be the inclusion.
    Then $j \circ r$ is homotopic to $\id_Y$ relative to $\widehat{X}_v$.
    Let $P_\ast$ be the chain homotopy induced by this homotopy.
    It vanishes on $j_\ast C_\ast(\widehat{X}_v)$.
    For every $q$-cell $\sigma$ of $\widehat{X}_e$, it satisfies
    $P_q(\sigma\times\{u\})=\pm(\sigma\times e)$ and
    $P_{q+1}(\sigma\times e)=0$.
    Thus, $P_k(a)=0$ because $a|_{E_u}=0$.
    Moreover, $P_k$ is norm non-increasing on $C_k(E_u)$.
    Consequently,
    \[
        \Mass{P_k(a-a')}
        =\Mass{P_k(a')}
        \leq\Mass{a'}.
    \]
    Since $a-a'$ is a cycle, we have
    $\partial P_k(a-a')=(j\circ r)_\ast(a-a')-(a-a')$.
    Moreover,
    $(j\circ r)_\ast(a-a')=j_\ast(r_\ast a-r_\ast a')=j_\ast\zeta$.
    Therefore, $H\defq j_\ast B-P_k(a-a')$ satisfies $\partial H=a-a'$ and
    $\Mass{H}\leq \Mass{B}+\Mass{a'}$, which is the claimed bound.
\end{proof}

\begin{rem}%
    \label{rem:leaf-collapse-preserves-tree-decomposition}
    In the setting of \cref{rem:leaf-collapse-decomposition}, suppose that
    $a'\in C_k(E_u)$ satisfies $\partial a'=b(e,u)$, and let $T'$ be obtained
    from $T$ by deleting $e$ and its leaf $v$. Then $z'\defq z-a+a'$ is a
    cycle supported over $T'$. Its $T'$-decomposition is obtained by
    replacing $c(u)$ with $c(u)+a'$.
    Moreover, $\Mass{a}+\Mass{z'}\leq\Mass{z}+\Mass{a'}$.
    Indeed, $a$ and $z-a$ have disjoint supports, while $z'=(z-a)+a'$.
\end{rem}

\begin{thm}%
    \label{thm:filling-tree-diagrams-kd}
    Let $k\geq 2$ and let $z \in C_k(\widehat{X})$ be a cycle with $\Mass{z} \leq n$.
    Assume that $z$ is supported in $\pi^{-1}(T)$ for a finite subtree $T\subseteq X$.
    Then $z$ bounds a cellular $(k+1)$-chain $H$ such that
    \[
        \Mass{H}\leq
        \overline{\max_e \Filling_{G_e}^{k}}(n)
        +\overline{\max_v \Filling_{G_v}^{k+1}}
        \bigl(n+2 \cdot \overline{\max_e \Filling_{G_e}^{k}}(n)\bigr),
    \]
    where the maxima are taken over all edges and vertices of $X$, respectively.
\end{thm}
\begin{proof}
    Let $r$ denote the number of edges in $T$ and fix a vertex $v_0$ of $T$.
    We use the iterative collapse procedure from
    the proofs of \cref{thm:filling tree diagrams,thm:filling-tree-diagrams-2d},
    which treat the homotopical cases $k=1,2$. Let $e_1,\ldots, e_r$ be the sequence of edges we use to collapse $T$ onto $v_0$ iteratively and let $z_0\defq z$.

    We write $e_i=\{u_i,v_i\}$ for the edge we remove in the $i$-th step by collapsing its leaf $v_i$, $z_{i-1}$ for the cycle we obtained from the $(i-1)$-th step, and $c_{i-1}$ and $b_{i-1}$ for the corresponding chains of the $T$-decomposition of $z_{i-1}$. Let $a_i\defq c_{i-1}(e_i)+c_{i-1}(v_i)$ and $\beta_i\defq b_{i-1}(e_i,u_i)$.
    By \cref{rem:leaf-collapse-decomposition,lem:tree collapse single cell-kd}, there are chains $a_i'$ and $H_i$ such that $\partial H_i=a_i-a_i'$.
    Setting $z_i\defq z_{i-1}-a_i+a_i'=z_{i-1}-\partial H_i$ yields the decomposition after the collapse of $e_i$ by \cref{rem:leaf-collapse-preserves-tree-decomposition}. In particular, $b_i(e_j,w)=b_0(e_j,w)$ for every edge $e_j$ of $T$ with $j>i$ and every vertex $w$ of such an edge $e_j$. After all $r$ edges have been collapsed, $z_r$ is supported in $\widehat{X}_{v_0}$ and $\partial(\sum_iH_i)=z_0-z_r$.

    Each $\beta_i$ is therefore equal to $b_0(e,w)$ for some edge $e$ of $T$
    and some vertex $w$ of $e$.
    By \cref{lem:tree-decomposition-properties},
    \[
        \sum_i\Mass{\beta_i}
        \leq \sum_{e\subseteq T}\Mass{c_0(e)}
        \leq \Mass{z_0}
        \leq n.
    \]
    Set $A\defq \sum_i\Mass{a_i'}$.
    By \cref{lem:tree collapse single cell-kd} and superadditivity, we obtain
    \[
        A
        \leq \overline{\max_e \Filling_{G_e}^{k}}
        \Bigl(\sum_i\Mass{\beta_i}\Bigr)
        \leq \overline{\max_e\Filling_{G_e}^{k}}(n).
    \]

    For the $i$-th collapse, the inequality in
    \cref{rem:leaf-collapse-preserves-tree-decomposition} gives
    \[\Mass{a_i}+\Mass{z_i} \leq \Mass{z_{i-1}}+\Mass{a_i'}.\]
    Summing these inequalities for all $i$ and cancelling the terms
    $\Mass{z_1},\ldots,\Mass{z_{r-1}}$ gives
    \[
        \Mass{z_r} + \sum_i\Mass{a_i}
        \leq \Mass{z_0}+\sum_i\Mass{a_i'}
        \leq n+A,
    \]
    hence $\Mass{z_r} + \sum_i\bigl(\Mass{a_i}+\Mass{a_i'}\bigr) \leq n+2A$.

    Choose a filling $H_0$ of $z_r$ in $\widehat{X}_{v_0}$ such that
    $\Mass{H_0} \leq \Filling_{G_{v_0}}^{k+1}(\Mass{z_r})$.
    By \cref{lem:tree collapse single cell-kd}, each $H_i$ with $1\leq i \leq r$ satisfies
    \(
    \Mass{H_i}
    \leq \Filling_{G_{v_i}}^{k+1}
    \bigl(\Mass{a_i}+\Mass{a_i'}\bigr)+\Mass{a_i'}.
    \)
    Set $H\defq H_0+\sum_iH_i$. Then
    \begin{align*}
        \Mass{H}
         &\leq \Mass{H_0}+ \sum\nolimits_i \Mass{H_i} \\
         &\leq \overline{\max_v \Filling_{G_v}^{k+1}}\Bigl(\Mass{z_r} + \sum\nolimits_i\bigl(\Mass{a_i}+\Mass{a_i'}\bigr)\Bigr) + A \\
         &\leq \overline{\max_v \Filling_{G_v}^{k+1}}(n+2A) + A \\
         &\leq \overline{\max_v \Filling_{G_v}^{k+1}}(n+2\cdot \overline{\max_e \Filling_{G_e}^{k}}(n)) + \overline{\max_e \Filling_{G_e}^{k}}(n).
    \end{align*}
    Finally, since $\partial H_0=z_r$ and $\partial(\sum_iH_i)=z-z_r$ we have $\partial H= z$. This completes the proof.
\end{proof}

\begin{proof}[Proof of \cref{thm:high-dimensional-Brown-trees}]
    This is now a direct consequence of \cref{thm:filling-tree-diagrams-kd}.
\end{proof}

\section{Special cases and applications}
\label{sec:applications}
In this section, we discuss some special cases and applications of our methods to group extensions and amalgamated products of nilpotent groups.

\subsection{Group extensions}

Let $1 \to H \to G \to Q \to 1$ be a group extension, where
$H$ and $Q$ are finitely presented.

Throughout, fix a finite presentation of $Q$ and let $X$ be the corresponding
Cayley complex. It becomes a $G$-polyhedral complex via the quotient map $G \onto Q$, and all stabilizers are $H$.
Moreover, fix symmetric generating sets $S_H$ and $S_G$ of $H$ and $G$, respectively, such that $S_H \subseteq S_G$.

Since all vertex and edge stabilizers of this action are equal to $H$ and
$\Dehn_X\approx\Dehn_Q$, \cref{thm:Dehn-Brown-complexes} yields the bound
\[
    \Dehn_G(n)
    \preccurlyeq
    \Dehn_Q(n)\cdot
    \Dehn_H\bigl(\overline{\dist_H^G}(\Dehn_Q(n))\bigr).
\]
This bound is not optimal in general; see \cref{rem:monodromy-bound-better-than-distortion}.
The goal of this subsection is to give a better upper bound
for group extensions by establishing a tighter control on the lengths of the paths in \cref{lem:tree collapse single cell}.

\begin{rem}%
    \label{rem:edge-elements-from-haefliger-data}
    Fix a choice of the data in \cref{rem:description-haefliger-construction}.
    If $e=(v,w)$ is an oriented edge of $\lquot{X}{G}$, define
    \[
        m_e \defq \transElem{w}{e}^{\vphantom{1}}\transElem{v}{e}^{-1}.
    \]
    With this convention, transporting a point $a \in \widehat{X}_v$
    along $\widehat{X}_e \times [0,1]$ results in the element $m_e^{-1} a m_e \in \widehat{X}_w$.
    We have $m_{\overline{e}}=m_e^{-1}$, where $\overline{e} = (w,v)$ is the same edge
    with the opposite orientation.
\end{rem}

\begin{defn}[Monodromy distortion]%
    \label{def:monodromy-distortion}
    An \emph{$\widehat{X}$-path} is a sequence
    \[
        \widehat{\gamma} = (h_0,e_1,h_1,\ldots,e_k,h_k),
    \]
    where $h_i \in H$ and the $e_i=(v_{i-1},v_i)$ are oriented edges in $X$
    which form an oriented path. It is called an \emph{$\widehat{X}$-loop} if
    $v_k=v_0$. Its length is defined as
    \[
        \ell_{\widehat{X}}(\widehat{\gamma}) \defq k + \sum_{i=0}^{k} \ell_H(h_i).
    \]

    For an oriented edge $e_i$ in $X$, let $m_{e_i}$ denote the element associated
    to the oriented edge $p(e_i)$ of $\lquot{X}{G}$ as in \cref{rem:edge-elements-from-haefliger-data}.
    In this way, we can associate to each $\widehat{X}$-loop $\widehat{\gamma}$ an element
    \[
        W(\widehat{\gamma}) \defq h_0 m_{e_1} h_1 \cdots m_{e_k}h_k \in H.
    \]
    The element $W(\widehat{\gamma})$ is obtained from the vertical pieces $h_i$ under the monodromy map defined by the path $e_1\cdots e_i$, that is, by pushing each $h_i$ to the fiber over $v_0$ along this path.

    We define the \emph{monodromy distortion} function $\MonDist_{\widehat{X}} \colon \mathbb{N}\to \mathbb{R}$ by
    \[
        \MonDist_{\widehat{X}}(n) \defq \max \{ \ell_H(W(\widehat{\gamma})) \mid \ell_{\widehat{X}}(\widehat{\gamma})\leq n\}.
    \]
\end{defn}

\begin{lem}%
    \label{lem:monodromy-well-defined}
    Let $\widehat{X}$ and $\widehat{X}'$ be constructed using two different choices of
    elements $\{\transElem{v}{e}\}$ and $\{\transElem{v}{e}'\}$.
    Then $\MonDist_{\widehat{X}} \approx \MonDist_{\widehat{X}'}$.
\end{lem}
\begin{proof}
    For every incidence $v\subset e$ in $\lquot{X}{G}$, set $a_{v,e}\defq \transElem{v}{e}(\transElem{v}{e}')^{-1}\in H$.
    If $e=(v,w)$, then the corresponding elements from \cref{rem:edge-elements-from-haefliger-data} satisfy
    \[
        m_e = a_{w,e} m'_e a_{v,e}^{-1}.
    \]
    Let $C \defq \max \{ \ell_H(a_{v,e}) \}$.
    Write $W_{\widehat{X}}$ and $W_{\widehat{X}'}$ for the words formed using
    the elements $m_e$ and $m'_e$, respectively.
    Given a $\widehat{X}'$-loop
    $\widehat{\gamma}=(h_0,e_1,h_1,\ldots,e_k,h_k)$, where
    $e_i=(v_{i-1},v_i)$, define an $\widehat{X}$-loop
    $\widehat{\eta}=(g_0,e_1,g_1,\ldots,e_k,g_k)$ by
    \[
        g_i \defq \begin{cases}
            h_0a_{v_1,e_1}^{-1}                          & i = 0; \\
            a_{v_{i-1},e_i} h_i a_{v_{i+1},e_{i+1}}^{-1} & 1 \leq i \leq k-1; \\
            a_{v_{k-1},e_k} h_k                          & i = k.
        \end{cases}
    \]
    Then $W_{\widehat{X}}(\widehat{\eta}) = W_{\widehat{X}'}(\widehat{\gamma})$ and
    \[
        \ell_{\widehat{X}}(\widehat{\eta})
        \leq
        \ell_{\widehat{X}'}(\widehat{\gamma}) + 2Ck
        \leq
        (2C+1)\ell_{\widehat{X}'}(\widehat{\gamma}).
    \]
    Hence, $\MonDist_{\widehat{X}'}(n)\leq \MonDist_{\widehat{X}}((2C+1)n)$.
    Because the argument is symmetric, we obtain $\MonDist_{\widehat{X}}\approx \MonDist_{\widehat{X}'}$.
\end{proof}

\begin{defn}
    The \emph{monodromy distortion} of the extension $1 \to H \to G \to Q \to 1$ is defined as $\MonDist_H^G \defq \MonDist_{\widehat{X}}$
    for some choice of elements $\{\transElem{v}{e}\}$.
\end{defn}

\begin{lem}%
    \label{lem:monodromy-estimate}
    Let $\widehat{\gamma}$ be an $\widehat{X}$-loop based at $v_0$.
    Then
    \[
        d_{\widehat{X}_{v_0}}\bigl(\widehat{\gamma}(0),\widehat{\gamma}(1)\bigr)
        \leq \MonDist_{\widehat{X}}\bigl(\ell_{\widehat{X}}(\widehat{\gamma})\bigr).
    \]
\end{lem}
\begin{proof}
    Let $\widehat{\gamma}=(h_0,e_1,h_1,\ldots,e_k,h_k)$, $n=\ell_{\widehat{X}}(\widehat{\gamma})$, and set
    $u_i \defq m_{e_1}\cdots m_{e_i}$. Note that $u_k \in H$, as $\widehat{\gamma}$ is a loop.
    The associated word $W(\widehat{\gamma})$ can be written as
    \begin{align}
        W(\widehat{\gamma})
         &= h_0m_{e_1}h_1\cdots m_{e_k}h_k\notag \\
         &= h_0 (u_1 h_1 u_1^{-1}) (u_2 h_2 u_2^{-1}) \cdots (u_k h_k u_k^{-1}) u_k.\label{eq:induced-word-by-conjugates}
    \end{align}
    The formula \eqref{eq:induced-word-by-conjugates} describes the path in
    $\widehat{X}_{v_0}$ obtained by pulling the vertical pieces $h_i$ back along
    $e_1,\ldots,e_i$ to the base vertex fiber.
    Hence, the word $W(\widehat{\gamma})$ labels a path in $\widehat{X}_{v_0}$ from
    $\widehat{\gamma}(0)$ to $\widehat{\gamma}(1)$.
    Therefore,
    \[
        d_{\widehat{X}_{v_0}}\bigl(\widehat{\gamma}(0),\widehat{\gamma}(1)\bigr)
        \leq \ell_H(W(\widehat{\gamma}))
        \leq \MonDist_{\widehat{X}}(n).\qedhere
    \]
\end{proof}

\begin{thm}%
    \label{thm:bound-group-extensions}
    Let $1 \to H \to G \to Q \to 1$ be a group extension, where $Q$ and $H$ are finitely presented.
    Then $\Dehn_G(n) \preccurlyeq \Dehn_Q(n) \cdot \Dehn_H\bigl(\overline{\MonDist_H^G}(\Dehn_Q(n))\bigr)$.
\end{thm}
\begin{proof}
    Let $S_H$ be the chosen finite generating set of $H$.
    Recall that, for an incidence $v\subset e$ in $\lquot{X}{G}$,
    $\transConj{v}{e}\colon H\to H$ denotes conjugation by $\transElem{v}{e}$, so that the attaching
    map $\transGlue{v}{e}\colon \widehat{X}_e\to \widehat{X}_v$ is induced by
    $\transConj{v}{e}$ on vertices. Set
    \[
        C \defq
        \max \bigl\{
        1,\ell_H(\transConj{v}{e}^{-1}(s))
        \mid v \subset e \subseteq \lquot{X}{G} \text{ and } s \in S_H \bigr\}.
    \]
    Then $\ell_H(\transConj{v}{e}^{-1}(h)) \leq C \ell_H(h)$ for every $h \in H$.
    Let $\gamma_{e_i}$ be the subpaths from the proof of \cref{thm:filling tree diagrams}
    and denote their endpoints by $a_i, b_i \in H$ identified as vertices in $\widehat{X}_{v}$.
    The endpoints of the paths $\gamma'_{e_i} \subset \widehat{X}_{e}$ are therefore
    $a'_i = \transGlue{v}{e}^{-1}(a_i)$ and $b'_i = \transGlue{v}{e}^{-1}(b_i)$, respectively.
    Then
    \begin{equation}%
        \label{eq:monodromy-upper-bound-trees}
        \ell(\gamma'_{e_i}) = d_{\widehat{X}_e}(a'_i,b'_i)
        \leq C \cdot d_{\widehat{X}_v}(a_i,b_i)
        \leq C \cdot \MonDist_{\widehat{X}}(\ell(\gamma_{e_i})),
    \end{equation}
    where the last inequality is due to \cref{lem:monodromy-estimate}.
    Replacing the inequality $\ell(\gamma'_{e_i}) \leq \edgeDist_X( D \cdot \ell(\gamma_{e_i}))$
    by \cref{eq:monodromy-upper-bound-trees} one obtains the same upper bound on fillings as in \cref{thm:filling tree diagrams}
    with $\overline{\edgeDist_X}$ replaced by $C \cdot \overline{\MonDist_{\widehat{X}}}$.
    One now proceeds as in the proof of \cref{thm:Dehn-Brown-complexes}, keeping
    in mind that $\Dehn_X \approx \Dehn_Q$.
\end{proof}

\begin{lem}%
    \label{lem:monodromy-distortion-subgroup-distortion}
    $\MonDist_H^G\preccurlyeq \dist_H^G$.
\end{lem}
\begin{proof}
    Let $C \defq \max \{ \ell_G(m_e) \mid e \in \lquot{X}{G} \}$ and
    $\widehat{\gamma} = (h_0,e_1,h_1,\ldots,e_k,h_k)$ be an $\widehat{X}$-loop with
    $\ell_{\widehat{X}}(\widehat{\gamma})\leq n$.
    Then $\ell_G(h_i)\leq \ell_H(h_i)$, so
    \[
        \ell_G(W(\widehat{\gamma}))
        \leq Ck + \sum_{i=0}^{k} \ell_G(h_i)
        \leq C \Bigl(k+\sum_{i=0}^{k}\ell_H(h_i)\Bigr)
        \leq Cn.
    \]
    As $W(\widehat{\gamma}) \in H$, we have $\ell_H(W(\widehat{\gamma}))\leq \dist_H^G(Cn)$.
\end{proof}

The following special case recovers a result on the Dehn function of central extensions by Conner.

\begin{cor}[\cite{conner}]%
    \label{cor:central-extension}
    Let $1 \to H \to G \to Q \to 1$ be a central group extension where $Q$ and $H$ are finitely presented.
    Then $\Dehn_G(n) \preccurlyeq \Dehn_Q(n)^3$.
\end{cor}
\begin{proof}
    Because $H$ is central, it is abelian, so $\Dehn_H$ is at most quadratic.
    Moreover, using the description of the words $W(\widehat{\gamma})$ from
    \cref{eq:induced-word-by-conjugates}, one sees that $\MonDist_H^G$ is linear,
    as the conjugation action by any element of $G$ on $H$ is trivial.
\end{proof}

\begin{rem}%
    \label{rem:monodromy-bound-better-than-distortion}
    In general, for an extension $1\to H\to G\to Q\to 1$, we have $\MonDist_H^G \not\approx \dist_H^G$. Indeed, the proof of \cref{cor:central-extension} shows that, in the case of a central extension, $\MonDist_H^G(n)\approx n$. However, there are central extensions of nilpotent groups with polynomial distortion of any integer degree $d\geq 1$.
\end{rem}

\begin{example}
    An example of an extension for which the bound in \cref{thm:bound-group-extensions} cannot be improved is the solvable group
    \[
        \Sol_3(\mathbb{Z})=\mathbb{Z}^2\rtimes_A \mathbb{Z}= \left\langle x, y, t \mid [x,y], txt^{-1}=x^2y, tyt^{-1}=xy\right\rangle
        \text{ for }
        A=\begin{pmatrix}2 & 1 \\ 1 & 1\end{pmatrix}
    \]
    Indeed, it is well known that the Dehn function of $\Sol_3(\mathbb{Z})$ is exponential and that the same is true for the distortion of $\mathbb{Z}^2$ in $\Sol_3(\mathbb{Z})$~\cite{ECHLPT-92}. Since \cref{lem:monodromy-distortion-subgroup-distortion} implies that $\MonDist_{\mathbb{Z}^2}^{\Sol_3(\mathbb{Z})}$ is bounded above by an exponential function, the optimality follows.

    Note that the discussion in the previous paragraph also implies that the growth rate of $\MonDist_{\mathbb{Z}^2}^{\Sol_3(\mathbb{Z})}$ is precisely exponential. One can also give a direct proof that $\MonDist_{\mathbb{Z}^2}^{\Sol_3(\mathbb{Z})}$ satisfies an exponential lower bound. For this, observe that the Bass--Serre tree $X$ of $\Sol_3(\mathbb{Z})$ is the real line with edges corresponding to unit intervals with integer endpoints. There is a single orbit of edges. Since $t$ acts by $x\mapsto x+1$, we can choose $m_e=t^{-1}$ for all positively oriented edges $e$ of $X$. Let $e_i$ be the edge from $i-1$ to $i$ and $\overline{e_i}$ the inverse edge from $i$ to $i-1$.

    With these choices, consider the $\widehat{X}$-loop
    \[
        \widehat{\gamma}=(h_0,e_1,h_1,\ldots, e_k,h_k,\overline{e_k},h_{k+1},\overline{e_{k-1}}, \ldots, \overline{e_1},h_{2k}),
    \]
    with $h_k=x$ and $h_i=1$ for $i\neq k$.
    Then $W(\widehat{\gamma})=t^{-k}xt^k=x^{f_{2k-1}}y^{-f_{2k}}$, where $f_i$ denotes the $i$-th Fibonacci number.
    In particular, $W(\widehat{\gamma})$ grows exponentially in $\ell_{\widehat{X}}(\widehat{\gamma})=2k+1$.
    We deduce that $\MonDist_{\mathbb{Z}^2}^{\Sol_3(\mathbb{Z})}(n) \succcurlyeq e^n$.
\end{example}

\subsection{Amalgams of nilpotent groups along cyclic subgroups}
Hidber proved in~\cite{Hidber-1999} that if $G=G_1\ast_H G_2$ is an amalgam of two nilpotent groups of class $\leq c$ along an infinite cyclic subgroup $H$, then $\dist_H^G(n)\preccurlyeq n^c$ and $\Dehn_G(n)\preccurlyeq n^{4c^2}$. Combining Hidber's distortion bound, the fact that Dehn functions of nilpotent groups of class $c$ are bounded by $n^{c+1}$~\cite{GerstenHoltRiley-2003}, and \cref{thm:Dehn-Brown-trees}, we obtain the following improved upper bound.
\begin{cor}
    Let $G=G_1\ast_H G_2$ be an amalgamated product of finitely generated nilpotent groups $G_1$ and $G_2$ along a cyclic subgroup $H$. Then $\Dehn_G(n)\preccurlyeq n^{c(c+1)+1}$.
\end{cor}

\bibliography{references}

\end{document}